\documentclass[12pt]{article}
\usepackage{amsfonts}
\usepackage{amsmath, amssymb, amsthm}
\usepackage{enumerate}
\usepackage{epsfig}
\usepackage{verbatim}

\newtheorem{theorem}{Theorem}

\newtheorem{definition}[theorem]{Definition}

\newtheorem{lemma}[theorem]{Lemma}

\newtheorem{proposition}[theorem]{Proposition}
\newtheorem{remark}[theorem]{Remark}

\input{tcilatex}
\begin{document}

\title{An $hp$ finite element method for a singularly perturbed
reaction-convection-diffusion boundary value problem with two small
parameters\\
}
\author{I. Sykopetritou and C. Xenophontos\thanks{
Corresponding author. Email: xenophontos@ucy.ac.cy} \\
Department of Mathematics and Statistics \\
University of Cyprus \\
P.O. Box 20537 \\
1678 Nicosia \\
Cyprus}
\maketitle

\begin{abstract}
We consider a second order singularly perturbed boundary value problem, of
reaction-convection-diffusion type with two small parameters, and the
approximation of its solution by the $hp$ version of the Finite Element
Method on the so-called {\emph{Spectral Boundary Layer}} mesh. We show that
the method converges uniformly, with respect to both singular perturbation
parameters, at an exponential rate when the error is measured in the energy
norm. Numerical examples are also presented, which illustrate our
theoretical findings.
\end{abstract}

\textbf{Keywords}: singularly perturbed problem;
reaction-convection-diffusion; boundary layers; $hp$ finite element method;
robust exponential convergence

\vspace{0.5cm}

\textbf{MSC2010}: 65N30

\section{Introduction}

\label{intro}

The numerical solution of singularly perturbed problems has been studied
extensively over the last few decades (see, e.g., the books \cite{mos}, \cite%
{morton}, \cite{rst} and the references therein). As is well known, a main
difficulty in these problems is the presence of \emph{boundary layers} in
the solution, whose accurate approximation independently of the singular
perturbation parameter(s), is of great importance for the overall quality of
the approximate solution to be considered reliable. In the context of the
Finite Element Method (FEM), the robust approximation of boundary layers
requires either the use of the $h$ version on non-uniform, layer-adapted
meshes (such as the Shishkin \cite{Shishkin2} or Bakhvalov \cite{B} mesh),
or the use of the high order $p$ and $hp$ versions on the so-called \emph{%
Spectral Boundary Layer} mesh \cite{MXO, ss}.

Usually, problems of convection-diffusion or reaction-diffusion type are
studied separately and several researchers have proposed and analyzed
numerical schemes for the robust approximation of their solution (see, e.g., 
\cite{rst} and the references therein). When there are two singular
perturbation parameters present in the differential equation, the problem
becomes reaction-convection-diffusion and the relatioship between the
parameters determines the `regime' we are in (see Table 1 ahead). In \cite{L}
this problem was addressed using the $h$ version of the FEM as well as
appropriate finite differences (see also \cite{BZ}, \cite{GORP}, \cite{LR}, 
\cite{ORPS}, \cite{RU}, \cite{TR}, \cite{TR2}). In the present article we
consider the $hp$ version of the FEM on the \emph{Spectral Boundary Layer}
mesh (from \cite{MXO}) and show that the method converges uniformly in the
perturbation parameters at an exponential rate, when the error is measured
in the energy norm.

The rest of the paper is organized as follows: in Section \ref{model} we
present the model problem and its regularity. Section \ref{mesh} presents
the discretization using the \emph{Spectral Boundary Layer} mesh and
contains our main result of uniform, exponential convergence. Finally, in
Section \ref{nr} we show the results of numerical computations that
illustrate and extend our theoretical findings.

With $I\subset \mathbb{R}$ an interval with boundary $\partial I$ and
measure $\left\vert I\right\vert $, we will denote by $C^{k}(I)$ the space
of continuous functions on $I$ with continuous derivatives up to order $k$.
We will use the usual Sobolev spaces $W^{k,m}(I)$ of functions on $\Omega $
with $0,1,2,...,k$ generalized derivatives in $L^{m}\left( I\right) $,
equipped with the norm and seminorm $\left\Vert \cdot \right\Vert _{k,m,I}$
and $\left\vert \cdot \right\vert _{k,m,I}\,$, respectively. When $m=2$, we
will write $H^{k}\left( I\right) $ instead of $W^{k,2}\left( I\right) $, and
for the norm and seminorm, we will write $\left\Vert \cdot \right\Vert
_{k,I} $ and $\left\vert \cdot \right\vert _{k,I}\,$, respectively. The
usual $L^{2}(I)$ inner product will be denoted by $\left\langle \cdot ,\cdot
\right\rangle _{I}$, with the subscript ommitted when there is no confusion.
We will also use the space 
\begin{equation*}
H_{0}^{1}\left( I\right) =\left\{ u\in H^{1}\left( I\right) :\left.
u\right\vert _{\partial \Omega }=0\right\} .
\end{equation*}%
The norm of the space $L^{\infty }(I)$ of essentially bounded functions is
denoted by $\Vert \cdot \Vert _{\infty ,I}$. Finally, the notation
\textquotedblleft $a\lesssim b$\textquotedblright\ means \textquotedblleft $%
a\leq Cb$\textquotedblright\ with $C$ being a generic positive constant,
independent of any parameters (e.g. discretization, singular perturbation,
etc.).


\section{The model problem and its regularity\label{model}}

We consider the following model problem (cf. \cite{omalley}): Find $u$ such
that 
\begin{eqnarray}
-\varepsilon _{1}u^{\prime \prime }(x)+\varepsilon _{2}b(x)u^{\prime
}(x)+c(x)u(x) &=&f(x):x\in I=\left( 0,1\right) ,  \label{de} \\
u(0)=u(1) &=&0\text{ },  \label{bc}
\end{eqnarray}%
where $0<\varepsilon _{1},\varepsilon _{2}\leq 1$ are given parameters that
can approach zero and the functions $b,c,f$ are given and sufficiently
smooth. \ In particular, we assume that they are analytic functions
satisfying, for some positive constants $\gamma _{f},\gamma _{c},\gamma _{b}$%
, independent of $\varepsilon _{1},\varepsilon _{2},$ 
\begin{equation}
\left\Vert f^{(n)}\right\Vert _{\infty ,I}\lesssim n!\gamma
_{f}^{n}\;,\;\left\Vert c^{(n)}\right\Vert _{\infty ,I}\lesssim n!\gamma
_{c}^{n}\;,\;\left\Vert b^{(n)}\right\Vert _{\infty ,I}\lesssim n!\gamma
_{b}^{n}\;\forall \;n=0,1,2,...  \label{analytic}
\end{equation}%
In addition, we assume that there exist constants $\beta ,\gamma ,\rho $,
independent of $\varepsilon _{1},\varepsilon _{2},$ such that $\forall
\;x\in \overline{I}$ 
\begin{equation}
b(x)\geq \beta >0\;,\;c(x)\geq \gamma >0\;,\;c(x)-\frac{\varepsilon _{2}}{2}%
b^{\prime }(x)\geq \rho >0.  \label{data}
\end{equation}%
The following result was established in \cite{IX} and it gives a bound in
terms of classical differentiability regularity.

\begin{proposition}
\label{thm_reg0}Let $u$ be the solution of (\ref{de}), (\ref{bc}). Then,
there exists a positive constant $K$, independent of $\varepsilon
_{1},\varepsilon _{2}$ and $u$, such that for $n=0,1,2,...$%
\begin{equation*}
\left\Vert u^{(n)}\right\Vert _{\infty ,I}\lesssim K^{n}\max \left\{
n,\varepsilon _{1}^{-1},\varepsilon _{2}^{-1}\right\} ^{n}.
\end{equation*}
\end{proposition}

More details arise if one studies the structure of the solution to (\ref{de}%
), which depends on the roots of the characteristic equation associated with
the differential operator. \ For this reason, we let $\lambda
_{0}(x),\lambda _{1}(x)$ be the solutions of the characteristic equation and
set%
\begin{equation}
\mu _{0}=-\underset{x\in \lbrack 0,1]}{\max }\lambda _{0}(x)\;,\;\mu _{1}=%
\underset{x\in \lbrack 0,1]}{\min }\lambda _{1}(x),  \label{mu}
\end{equation}%
or equivalently,%
\begin{equation*}
\mu _{0,1}=\underset{x\in \lbrack 0,1]}{\min }\frac{\mp \varepsilon _{2}b(x)+%
\sqrt{\varepsilon _{2}^{2}b^{2}(x)+4\varepsilon _{1}c(x)}}{2\varepsilon _{1}}%
.
\end{equation*}%
The following hold true \cite{RU,TR}:%
\begin{equation}
\left. 
\begin{array}{c}
1<<\mu _{0}\leq \mu _{1\;\;},\;\frac{\varepsilon _{2}}{\varepsilon
_{2}+\varepsilon _{1}^{1/2}}\lesssim \varepsilon _{2}\mu _{0}\lesssim
1\;,\;\varepsilon _{1}^{1/2}\mu _{0}\lesssim 1 \\ 
\max \{\mu _{0}^{-1},\varepsilon _{1}\mu _{1}\}\lesssim \varepsilon
_{1}+\varepsilon _{2}^{1/2}\;,\;\varepsilon _{2}\lesssim \varepsilon _{1}\mu
_{1} \\ 
\text{for }\varepsilon _{2}^{2}\geq \varepsilon _{1}:\;\varepsilon
_{1}^{-1/2}\lesssim \mu _{1}\lesssim \varepsilon _{1}^{-1} \\ 
\text{for }\varepsilon _{2}^{2}\leq \varepsilon _{1}:\;\varepsilon
_{1}^{-1/2}\lesssim \mu _{1}\lesssim \varepsilon _{1}^{-1/2}%
\end{array}%
\right\} .  \label{mu_a}
\end{equation}%
The values of $\mu _{0},\mu _{1}$ determine the strength of the boundary
layers and since $\left\vert \lambda _{0}(x)\right\vert <\left\vert \lambda
_{1}(x)\right\vert $ the layer at $x=1$ is stronger than the layer at $x=0$.
\ Essentially, there are three regimes \cite{L}, as seen in Table 1.

\begin{table}[h]
\begin{center}
\begin{tabular}{||cccc||}
\hline
\label{table1} &  & $\mu _{0}$ & $\mu _{1}$ \\[0.5ex] \hline\hline
convection-diffusion & $\varepsilon _{1}<<\varepsilon _{2}=1$ & $1$ & $%
\varepsilon _{1}^{-1}$ \\ \hline
convection-reaction-diffusion & $\varepsilon _{1}<<\varepsilon _{2}^{2}<<1$
& $\varepsilon _{2}^{-1}$ & $\varepsilon _{2}/\varepsilon _{1}$ \\ \hline
reaction-diffusion & $1>>\varepsilon _{1}>>\varepsilon _{2}^{2}$ & $%
\varepsilon _{1}^{-1/2}$ & $\varepsilon _{1}^{-1/2}$ \\[1ex] \hline
\end{tabular}%
\end{center}
\caption{Different regimes based on the relationship between $\protect%
\varepsilon_1$ and $\protect\varepsilon_2$.}
\end{table}

The above considerations suggest the following two cases:

\begin{enumerate}
\item $\varepsilon _{1}$ is large compared to $\varepsilon _{2}$: this is
similar to a `regular perturbation' of reaction-diffusion type. If one
considers the limiting case $\varepsilon _{2}=0$, then one sees that there
are two boundary layers, one at each endpoint, of width $O\left( \varepsilon
_{1}^{1/2}\right) $. \ This situation has been studied in the literature
(see, e.g., \cite{melenk}) and will not be considered further in this
article.

\item $\varepsilon _{1}$ is small compared to $\varepsilon _{2}$: before
discussing the different regimes, it is instructive to consider the limiting
case $\varepsilon _{1}=0$. Then there is an exponential layer (of length
scale $O(\varepsilon _{2})$) at the \emph{left} endpoint. The homogeneous
equation (with constant coefficients) suggests that the different regimes
are $\varepsilon_{1}<<\varepsilon _{2}^{2},\varepsilon _{1}\approx
\varepsilon _{2}^{2},\varepsilon _{1}>>\varepsilon _{2}^{2}.$

\begin{enumerate}
\item In the regime $\varepsilon _{1}<<\varepsilon _{2}^{2}$ we have $\mu
_{0}=O(\varepsilon _{2}^{-1})$ and $\mu _{1}=O(\varepsilon _{2}\varepsilon
_{1}^{-1})$. Hence $\mu _{1}$ is much larger than $\mu _{0}$ and the
boundary layer in the vicinity of $x=1$ is stronger. \ Consequently, there
is a layer of width $O(\varepsilon _{2})$ at the left endpoint (the one that
arose from the analysis of the case $\varepsilon _{1}=0$) and additionally,
there is another layer at the right endpoint, of width $O(\varepsilon
_{1}/\varepsilon _{2})$.

\item In the regime $\varepsilon _{1}\approx \varepsilon _{2}^{2}$ there are
layers at both endpoints of width $O(\varepsilon _{2})=O\left( \varepsilon
_{1}^{1/2}\right) $.

\item In the regime $\varepsilon _{2}^{2}<<\varepsilon _{1}<<1$, there are
layers at both endpoints of width $O\left( \varepsilon _{1}^{1/2}\right) $.
\end{enumerate}
\end{enumerate}

\subsection{The asymptotic expansion\label{asy}}

We focus on Case 2 (a)--(c) above, i.e. $\varepsilon _{1}<\varepsilon _{2}$,
and describe an appropriate asymptotic expansion for $u$, in what follows.
(The material also appears in \cite{IX}.)

\subsubsection{The regime $\protect\varepsilon _{1}<<\protect\varepsilon %
_{2}^{2}<<1\label{regime1}$}

In this case we anticipate a layer of width $O(\varepsilon _{2})$ at the
left endpoint and a layer of width $O\left( \varepsilon _{1}/\varepsilon
_{2}\right) $ at the right endpoint. To deal with this we define the \emph{%
stretched variables} $\tilde{x}=x/\varepsilon _{2}$ and $\hat{x}%
=(1-x)\varepsilon _{2}/\varepsilon _{1}$ and make the formal ansatz 
\begin{equation}
u\sim \sum_{i=0}^{\infty }\sum_{j=0}^{\infty }\varepsilon
_{2}^{i}(\varepsilon _{1}/\varepsilon _{2}^{2})^{j}\left( u_{i,j}(x)+\tilde{u%
}_{i,j}^{BL}(\tilde{x})+\hat{u}_{i,j}^{BL}(\hat{x})\right) ,  \label{c1}
\end{equation}%
with $u_{i,j},\tilde{u}_{i,j}^{BL},\hat{u}_{i,j}^{BL}$ to be determined.
Substituting (\ref{c1}) into (\ref{de}), separating the slow (i.e. $x$) and
fast (i.e. $\tilde{x},\hat{x}$) variables, and equating like powers of $%
\varepsilon _{1}$ and $\varepsilon _{2}$ , we get{\footnote{%
The constant coefficient case is considerably simpler -- see \cite{Irene}.}} 
\begin{equation}
\left. 
\begin{array}{c}
u_{0,0}(x)=\frac{f(x)}{c(x)} \\ 
u_{i,0}(x)=-\frac{b(x)}{c(x)}u_{i-1,0}^{\prime }(x),i\geq 1 \\ 
u_{0,j}(x)=u_{1,j}(x)=0,j\geq 1 \\ 
u_{i,j}(x)=\frac{1}{c(x)}\left( u_{i-2,j-1}^{\prime \prime
}(x)-b(x)u_{i-1,j}^{\prime }(x)\right) ,i\geq 2,j\geq 1%
\end{array}%
\right\} ,  \label{c1smooth}
\end{equation}%
\begin{equation}
\left. 
\begin{array}{c}
\tilde{b}_{0}\left( \tilde{u}_{0,0}^{BL}\right) ^{\prime }+\tilde{c}_{0}%
\tilde{u}_{0,0}^{BL}=0 \\ 
\tilde{b}_{0}\left( \tilde{u}_{i,0}^{BL}\right) ^{\prime }+\tilde{c}_{0}%
\tilde{u}_{i,0}^{BL}=-\sum_{k=1}^{i}\left( \tilde{b}_{k}\left( \tilde{u}%
_{i-k,0}^{BL}\right) ^{\prime }+\tilde{c}_{k}\tilde{u}_{i-k,0}^{BL}\right)
,i\geq 1 \\ 
\tilde{b}_{0}\left( \tilde{u}_{0,j}^{BL}\right) ^{\prime }+\tilde{c}_{0}%
\tilde{u}_{0,j}^{BL}=\left( \tilde{u}_{0,j-1}^{BL}\right) ^{\prime \prime
},j\geq 1 \\ 
\tilde{b}_{0}\left( \tilde{u}_{i,j}^{BL}\right) ^{\prime }+\tilde{c}_{0}%
\tilde{u}_{i,j}^{BL}=\left( \tilde{u}_{i,j-1}^{BL}\right) ^{\prime \prime
}-\sum_{k=1}^{i}\left( \tilde{b}_{k}\left( \tilde{u}_{i-k,j}^{BL}\right)
^{\prime }+\tilde{c}_{k}\tilde{u}_{i-k,j}^{BL}\right) ,i\geq 1,j\geq 1%
\end{array}%
\right\} ,  \label{c1BLa}
\end{equation}%
\begin{equation}
\left. 
\begin{array}{c}
\left( \hat{u}_{i,0}^{BL}\right) ^{\prime \prime }+\hat{b}_{0}\left( \hat{u}%
_{i,0}^{BL}\right) ^{\prime }=0,i\geq 0 \\ 
\left( \hat{u}_{0,j}^{BL}\right) ^{\prime \prime }+\hat{b}_{0}\left( \hat{u}%
_{0,j}^{BL}\right) ^{\prime }=\hat{c}_{0}\hat{u}_{0,j-1}^{BL},j\geq 1 \\ 
\left( \hat{u}_{i,1}^{BL}\right) ^{\prime \prime }+\hat{b}_{0}\left( \hat{u}%
_{i,1}^{BL}\right) ^{\prime }=\hat{c}_{0}\hat{u}_{i,0}^{BL}-\hat{b}%
_{1}\left( \hat{u}_{i-1,0}^{BL}\right) ^{\prime },i\geq 1 \\ 
\left( \hat{u}_{1,j}^{BL}\right) ^{\prime \prime }+\hat{b}_{0}\left( \hat{u}%
_{1,j}^{BL}\right) ^{\prime }=\hat{c}_{0}\hat{u}_{1,j-1}^{BL}-\hat{b}%
_{1}\left( \hat{u}_{0,j-1}^{BL}\right) ^{\prime }+\hat{c}_{1}\hat{u}%
_{0,j-2}^{BL},j\geq 2 \\ 
\left( \hat{u}_{i,j}^{BL}\right) ^{\prime \prime }+\hat{b}_{0}\left( \hat{u}%
_{i,j}^{BL}\right) ^{\prime }=\hat{c}_{0}\hat{u}_{i,j-1}^{BL}-\hat{b}%
_{j}\left( \hat{u}_{i-j,0}^{BL}\right) ^{\prime }+ \\ 
\sum_{k=1}^{j-1}\left\{ -\hat{b}_{k}\left( \hat{u}_{i-k,j-k}^{BL}\right)
^{\prime }+\hat{c}_{k}\hat{u}_{i-k,j-k-1}^{BL}\right\} ,i\geq 2,j=2,...,i \\ 
\left( \hat{u}_{i,j}^{BL}\right) ^{\prime \prime }+\hat{b}_{0}\left( \hat{u}%
_{i,j}^{BL}\right) ^{\prime }=\hat{c}_{0}\hat{u}_{i,j-1}^{BL}+ \\ 
\sum_{k=1}^{i}\left\{ -\hat{b}_{k}\left( \hat{u}_{i-k,j-k}^{BL}\right)
^{\prime }+\hat{c}_{k}\hat{u}_{i-k,j-k-1}^{BL}\right\} ,i\geq 2,j>i%
\end{array}%
\right\} ,  \label{c1BLb}
\end{equation}%
where the notation $\tilde{b}_{k}(\tilde{x})=\tilde{x}^{k}b^{(k)}(0)/k!$ , $%
\hat{b}_{k}(\hat{x})=(-1)^{k}\hat{x}^{k}b^{(k)}(1)/k!$ is used, and
analogously for the other terms. (We also adopt the convention that empty
sums are 0.) The BVPs (\ref{c1BLa})--(\ref{c1BLb}) are supplemented with the
following boundary conditions (in order for (\ref{bc}) to be satisfied) for
all $i,j\geq 0$: 
\begin{equation}
\left. 
\begin{array}{c}
\tilde{u}_{i,j}^{BL}(0)=-u_{i,j}(0)\;,\;\lim_{\tilde{x}\rightarrow \infty }%
\tilde{u}_{i,j}^{BL}(\tilde{x})=0 \\ 
\hat{u}_{i,j}^{BL}(0)=-u_{i,j}(1)\;,\;\lim_{\hat{x}\rightarrow \infty }\hat{u%
}_{i,j}^{BL}(\hat{x})=0%
\end{array}%
\right\} .  \label{c1BC}
\end{equation}%
Next, we define for some $M_{1},M_{2}\in \mathbb{N},$%
\begin{eqnarray}
u_{M}(x) &:&=\sum_{i=0}^{M_{1}}\sum_{j=0}^{M_{2}}\varepsilon
_{2}^{i}(\varepsilon _{1}/\varepsilon _{2}^{2})^{j}u_{i,j}(x),  \label{uM1}
\\
\tilde{u}_{M}^{BL}(\tilde{x})
&:&=\sum_{i=0}^{M_{1}}\sum_{j=0}^{M_{2}}\varepsilon _{2}^{i}(\varepsilon
_{1}/\varepsilon _{2}^{2})^{j}\tilde{u}_{i,j}^{BL}(\tilde{x}),  \label{uBL1a}
\\
\hat{u}_{M}^{BL}(\hat{x})
&:&=\sum_{i=0}^{M_{1}}\sum_{j=0}^{M_{2}}\varepsilon _{2}^{i}(\varepsilon
_{1}/\varepsilon _{2}^{2})^{j}\hat{u}_{i,j}^{BL}(\hat{x}),  \label{uBL1b} \\
r_{M}^{1} &:&=u-\left( u_{M}+\tilde{u}_{M}^{BL}+\hat{u}_{M}^{BL}\right)
\label{rM1}
\end{eqnarray}%
and we have the following decomposition 
\begin{equation}
u=u_{M}+\tilde{u}_{M}^{BL}+\hat{u}_{M}^{BL}+r_{M}^{1}.  \label{decomp1}
\end{equation}%
The following was shown in \cite{IX} and it gives analytic regularity bounds
on each term in the decomposition (\ref{decomp1}).

\begin{proposition}
\label{thm_reg1}Assume (\ref{analytic}), (\ref{data}) hold. Then there exist
positive constants $K_{1},K_{2},\tilde{K},\hat{K},\tilde{\gamma},\hat{\gamma}%
,\delta $, independent of $\varepsilon _{1},\varepsilon _{2},$ such that the
solution $u$ of (\ref{de})--(\ref{bc}) can be decomposed as in (\ref{decomp1}%
), with 
\begin{equation}
\left\Vert u_{M}^{(n)}\right\Vert _{\infty ,I}\lesssim n!K_{1}^{n}\;\forall
\;n\in \mathbb{N}_{0},  \label{uM1bound}
\end{equation}%
\begin{equation}
\left\vert \left( \tilde{u}_{M}^{BL}\right) ^{(n)}(x)\right\vert \lesssim 
\tilde{K}^{n}\varepsilon _{2}^{-n}e^{-dist(x,\partial I)/\varepsilon
_{2}}\;\forall \;n\in \mathbb{N}_{0},  \label{uBL1abound}
\end{equation}%
\begin{equation}
\left\vert \left( \hat{u}_{M}^{BL}\right) ^{(n)}(x)\right\vert \lesssim \hat{%
K}_{2}^{n}\left( \frac{\varepsilon _{1}}{\varepsilon _{2}}\right)
^{-n}e^{-dist(x,\partial I)\varepsilon _{2}/\varepsilon _{1}}\;\forall
\;n\in \mathbb{N}_{0},  \label{uBL1bbound}
\end{equation}%
\begin{equation}
\left\Vert r_{M}\right\Vert _{\infty ,\partial I}+\left\Vert
r_{M}\right\Vert _{0,I}+\varepsilon _{1}^{1/2}\left\Vert r_{M}^{\prime
}\right\Vert _{0,I}\lesssim \max \{e^{-\delta \varepsilon _{2}/\varepsilon
_{1}},e^{-\delta /\varepsilon _{2}}\}  \label{rM1bound}
\end{equation}%
provided $4\varepsilon _{2}e^{2}M_{1}^{2}\max \{1,K_{2},\tilde{\gamma}_{1},%
\tilde{\gamma}_{2},\hat{\gamma}_{1},\hat{\gamma}_{2},\tilde{\gamma}%
_{1}^{2}\}<1$ and $\frac{\varepsilon _{1}}{\varepsilon _{2}^{2}}%
e^{2}M_{2}\max \{1,K_{2},\tilde{\gamma}_{1},\tilde{\gamma}_{2},\hat{\gamma}%
_{1},\hat{\gamma}_{2},\tilde{\gamma}_{1}^{2}\}<1$.
\end{proposition}

\subsubsection{The regime $\protect\varepsilon _{1}\approx \protect%
\varepsilon _{2}^{2}\label{regime2}$}

Now there are layers at both endpoints of width $O(\varepsilon _{2})$. So
with $\tilde{x}=x/\varepsilon _{2},\overline{x}=(1-x)/\varepsilon _{2},$ we
make the formal ansatz 
\begin{equation}
u\sim \sum_{i=0}^{\infty }\varepsilon _{2}^{i}\left( u_{i}(x)+\tilde{u}%
_{i}^{BL}(\tilde{x})+\overline{u}_{i}^{BL}(\overline{x})\right) ,  \label{c2}
\end{equation}%
with $u_{i},\tilde{u}_{i}^{BL},\overline{u}_{i}^{BL}$ to be determined. \
Subsituting (\ref{c2}) into (\ref{de}), separating the slow (i.e. $x$) and
fast (i.e. $\tilde{x},\overline{x})$ variables, and equating like powers of $%
\varepsilon _{1}(=\varepsilon _{2}^{2})$ and $\varepsilon _{2}$ we get 
\begin{equation*}
\left. 
\begin{array}{c}
u_{0}(x)=\frac{f(x)}{c(x)}\;,\;u_{1}(x)=-\frac{b(x)}{c(x)}u_{0}^{\prime }(x),
\\ 
u_{i}(x)=\frac{1}{c(x)}\left( u_{i-2}^{\prime \prime
}(x)-b(x)u_{i-1}^{\prime }(x)\right) ,i\geq 2,%
\end{array}%
\right\}
\end{equation*}%
\begin{equation*}
\left. 
\begin{array}{c}
-\left( \tilde{u}_{0}^{BL}\right) ^{\prime \prime }+\tilde{b}_{0}\left( 
\tilde{u}_{0}^{BL}\right) ^{\prime }+\tilde{c}_{0}\tilde{u}_{0}^{BL}=0, \\ 
-\left( \tilde{u}_{i}^{BL}\right) ^{\prime \prime }+\tilde{b}_{0}\left( 
\tilde{u}_{i}^{BL}\right) ^{\prime }+\tilde{c}_{0}\tilde{u}%
_{i}^{BL}=-\sum_{k=1}^{i}\left( \tilde{b}_{k}\left( \tilde{u}%
_{i-k}^{BL}\right) ^{\prime }+\tilde{c}_{k}\tilde{u}_{i-k}^{BL}\right)
,i\geq 1%
\end{array}%
\right\}
\end{equation*}%
\begin{equation*}
\left. 
\begin{array}{c}
-\left( \bar{u}_{i}^{BL}\right) ^{\prime \prime }-\bar{b}_{0}\left( \bar{u}%
_{i}^{BL}\right) ^{\prime }+\bar{c}_{0}\bar{u}_{i}^{BL}=0, \\ 
-\left( \bar{u}_{i}^{BL}\right) ^{\prime \prime }+\bar{b}_{0}\left( \bar{u}%
_{i}^{BL}\right) ^{\prime }+\bar{c}_{0}\bar{u}_{i}^{BL}=\sum_{k=1}^{i}\left( 
\bar{b}_{k}\left( \bar{u}_{i-k}^{BL}\right) ^{\prime }-\bar{c}_{k}\bar{u}%
_{i-k}^{BL}\right) ,i\geq 1%
\end{array}%
\right\}
\end{equation*}%
where the notation $\tilde{b}_{k}(\tilde{x})=\tilde{x}^{k}b^{(k)}(0)/k!$
etc., is used again. The above equations are supplemented with the following
boundary conditions (in order to satisfy (\ref{bc})):

\begin{equation*}
\left. 
\begin{array}{c}
u_{i}(0)+\tilde{u}_{i}^{BL}(0)=0, \\ 
u_{i}(1)+\bar{u}_{i}^{BL}(0)=0, \\ 
\lim_{\tilde{x}\rightarrow \infty }\tilde{u}_{i}^{BL}(\tilde{x})=0\,,\,\lim_{%
\overline{x}\rightarrow \infty }\bar{u}_{i}^{BL}(\overline{x})=0.%
\end{array}%
\right\}
\end{equation*}%
We then define, for some $M\in \mathbb{Z},$ 
\begin{equation*}
u_{M}(x)=\sum_{i=0}^{M}\varepsilon _{2}^{i}u_{i}(x),\tilde{u}_{M}^{BL}(%
\tilde{x})=\sum_{i=0}^{M}\varepsilon _{2}^{i}\tilde{u}_{i}^{BL}(\tilde{x}),%
\overline{u}_{M}^{BL}(\overline{x})=\sum_{i=0}^{M}\varepsilon _{2}^{i}\bar{u}%
_{i}^{BL}(\overline{x}),
\end{equation*}%
as well as 
\begin{equation}
u=u_{M}+\tilde{u}_{M}^{BL}+\hat{u}_{M}^{BL}+r_{M}.  \label{decomp2}
\end{equation}%
The following was proven in \cite{melenk97}.

\begin{proposition}
\label{thm_reg2}Assume (\ref{analytic}), (\ref{data}) hold. Then there exist
positive constants $K_{1},K_{2},\tilde{K},\overline{K},\delta $, independent
of $\varepsilon _{1},\varepsilon _{2},$ such that the solution $u$ of (\ref%
{de})--(\ref{bc}) can be decomposed as in (\ref{decomp2}), with 
\begin{equation*}
\left\Vert u_{M}^{(n)}\right\Vert _{\infty ,I}\lesssim n!K_{1}^{n}\;\forall
\;n\in \mathbb{N}_{0},
\end{equation*}%
\begin{equation*}
\left\vert \left( \tilde{u}_{M}^{BL}\right) ^{(n)}(x)\right\vert \lesssim 
\tilde{K}^{n}\varepsilon _{2}^{-n}e^{-dist(x,\partial I)/\varepsilon
_{2}}\;\forall \;n\in \mathbb{N}_{0},
\end{equation*}%
\begin{equation*}
\left\vert \left( \bar{u}_{M}^{BL}\right) ^{(n)}(x)\right\vert \lesssim 
\overline{K}^{n}\varepsilon _{2}^{-n}e^{-dist(x,\partial I)/\varepsilon
_{2}}\;\forall \;n\in \mathbb{N}_{0},
\end{equation*}%
\begin{equation*}
\left\Vert r_{M}\right\Vert _{\infty ,\partial I}+\left\Vert
r_{M}\right\Vert _{0,I}+\varepsilon _{1}^{1/2}\left\Vert r_{M}^{\prime
}\right\Vert _{0,I}\lesssim ,e^{-\delta /\varepsilon _{2}}
\end{equation*}%
provided $\varepsilon _{2}K_{2}M<1$.
\end{proposition}

\subsubsection{The regime $\protect\varepsilon _{2}^{2}<<\protect\varepsilon %
_{1}<<1\label{regime3}$}

We anticipate layers at both endpoints of width $O\left( \sqrt{\varepsilon
_{1}}\right) $. So we define the \emph{stretched variables} $\check{x}=x/%
\sqrt{\varepsilon _{1}}$ and $\breve{x}=(1-x)/\sqrt{\varepsilon _{1}}$ and
make the formal ansatz 
\begin{equation}
u\sim \sum_{i=0}^{\infty }\sum_{j=0}^{\infty }\varepsilon _{1}^{i/2}\left(
\varepsilon _{2}/\sqrt{\varepsilon _{1}}\right) ^{j}\left( u_{i,j}(x)+\check{%
u}_{i,j}^{BL}(\check{x})+\breve{u}_{i,j}^{BL}(\breve{x})\right) ,  \label{c3}
\end{equation}%
with $u_{i,j},\check{u}_{i,j}^{BL},\breve{u}_{i,j}^{BL}$ to be determined.
Substituting (\ref{c3}) into (\ref{de}), separating the slow (i.e. $x$) and
fast (i.e. $\check{x},\breve{x})$ variables, and equating like powers of $%
\varepsilon _{1}$ and $\varepsilon _{2}$ we get 
\begin{equation*}
\left. 
\begin{array}{c}
u_{0,0}(x)=\frac{f(x)}{c(x)},u_{1,0}(x)=u_{0,j}(x)=0,j\geq 1, \\ 
u_{i,0}(x)=\frac{1}{c(x)}u_{i-2,0}^{\prime \prime }(x),i\geq 2, \\ 
u_{2i+1,0}(x)=0,i\geq 1, \\ 
u_{1,1}(x)=-\frac{b(x)}{c(x)}u_{0,0}^{\prime }(x),u_{1,j}(x)=0,j\geq 2, \\ 
u_{i,j}(x)=\frac{1}{c(x)}\left( u_{i-2,j}^{\prime \prime
}(x)-b(x)u_{i-1,j-1}^{\prime }(x)\right) ,i\geq 2,j\geq 1,%
\end{array}%
\right\}
\end{equation*}%
\begin{equation*}
\left. 
\begin{array}{c}
-\left( \check{u}_{0,0}^{BL}\right) ^{\prime \prime }+\check{c}_{0}\check{u}%
_{0,0}^{BL}=0, \\ 
-\left( \check{u}_{i,0}^{BL}\right) ^{\prime \prime }+\check{c}_{0}\check{u}%
_{i,0}^{BL}=-\sum_{k=i}^{i}\check{c}_{k}\check{u}_{i-k,0}^{BL},i\geq 1 \\ 
-\left( \check{u}_{0,j}^{BL}\right) ^{\prime \prime }+\check{c}_{0}\check{u}%
_{0,j}^{BL}=-\check{b}_{0}\left( \check{u}_{0,j-1}^{BL}\right) ^{\prime
},j\geq 1 \\ 
-\left( \check{u}_{i,j}^{BL}\right) ^{\prime \prime }+\check{c}_{0}\check{u}%
_{i,j}^{BL}=-\check{b}_{0}\left( \check{u}_{i,j-1}^{BL}\right) ^{\prime }-
\\ 
\sum_{k=1}^{i}\left\{ \check{b}_{k}\left( \check{u}_{i-k,j-1}^{BL}\right)
^{\prime }+\check{c}_{k}\check{u}_{i-k,j}^{BL}\right\} ,i\geq 1,j\geq 1,%
\end{array}%
\right\}
\end{equation*}%
\begin{equation*}
\left. 
\begin{array}{c}
-\left( \grave{u}_{0,0}^{BL}\right) ^{\prime \prime }+\grave{c}_{0}\grave{u}%
_{0,0}^{BL}=0, \\ 
-\left( \grave{u}_{i,0}^{BL}\right) ^{\prime \prime }+\grave{c}_{0}\grave{u}%
_{i,0}^{BL}=-\sum_{k=1}^{i}\grave{c}_{k}\grave{u}_{i-k,0}^{BL},i\geq 1, \\ 
-\left( \grave{u}_{0,j}^{BL}\right) ^{\prime \prime }+\grave{c}_{0}\grave{u}%
_{0,j}^{BL}=\grave{b}_{0}\grave{u}_{0,j-1}^{BL},j\geq 1, \\ 
-\left( \grave{u}_{i,j}^{BL}\right) ^{\prime \prime }+\grave{c}_{0}\grave{u}%
_{i,j}^{BL}=\left( \grave{b}_{0}\grave{u}_{i,j-1}^{BL}\right) ^{\prime }- \\ 
\sum_{k=1}^{i}\left\{ \grave{b}_{k}\left( \grave{u}_{i-k,j-1}^{BL}\right)
^{\prime }-\grave{c}_{k}\grave{u}_{i-k,j}^{BL}\right\} ,i\geq 1,j\geq 1,%
\end{array}%
\right\}
\end{equation*}%
where the notation $\check{b}_{k}(\check{x})=\check{x}^{k}b^{(k)}(0)/k!$
etc., is used once more. The above equations are supplemented with the
following boundary conditions (in order to satisfy (\ref{bc})):

\begin{equation*}
\left. 
\begin{array}{c}
\check{u}_{i,j}^{BL}(0)=-u_{i,j}(0)\,,\,\grave{u}_{i,j}^{BL}(0)=-u_{i,j}(1),
\\ 
\lim_{\check{x}\rightarrow \infty }\check{u}_{i,j}^{BL}(\check{x}%
)=0\,,\,\lim_{\grave{x}\rightarrow \infty }\grave{u}_{i,j}^{BL}(\grave{x})=0.%
\end{array}%
\right\}
\end{equation*}%
We then define, for some $M\in \mathbb{Z},$%
\begin{eqnarray*}
u_{M}(x) &=&\sum_{i=0}^{M}\sum_{j=0}^{M}\varepsilon _{1}^{i/2}\left(
\varepsilon _{2}/\sqrt{\varepsilon _{1}}\right) ^{j}u_{i,j}(x), \\
\check{u}_{M}^{BL}(\check{x}) &=&\sum_{i=0}^{M}\sum_{j=0}^{M}\varepsilon
_{1}^{i/2}\left( \varepsilon _{2}/\sqrt{\varepsilon _{1}}\right) ^{j}\check{u%
}_{i,j}^{BL}(\check{x}), \\
\breve{u}_{M}^{BL}(\breve{x}) &=&\sum_{i=0}^{M}\sum_{j=0}^{M}\varepsilon
_{1}^{i/2}\left( \varepsilon _{2}/\sqrt{\varepsilon _{1}}\right) ^{j}\breve{u%
}_{i,j}^{BL}(\breve{x}),
\end{eqnarray*}%
and we have the following decomposition:%
\begin{equation}
u=u_{M}+\check{u}_{M}^{BL}+\breve{u}_{M}^{BL}+r_{M}.  \label{decomp3}
\end{equation}%
The theorem that follows is the analog of Theorem \ref{thm_reg1} \cite{XI}.

\begin{proposition}
\label{thm_reg3}Assume (\ref{analytic}), (\ref{data}) hold. Then there exist
positive constants $K_{1},\check{K},\breve{K},K_{2}$ and $\delta $,
independent of $\varepsilon _{1},\varepsilon _{2},$ such that the solution $%
u $ of (\ref{de})--(\ref{bc}) can be decomposed as in (\ref{decomp3}), with%
\begin{equation*}
\left\Vert u_{M}^{(n)}\right\Vert _{\infty ,I}\lesssim n!K_{1}^{n}\;\forall
\;n\in \mathbb{N}_{0},
\end{equation*}%
\begin{equation*}
\left\vert \left( \check{u}_{M}^{BL}\right) ^{(n)}(x)\right\vert \lesssim 
\check{K}^{n}\varepsilon _{1}^{-n/2}e^{-dist(x,\partial I)/\sqrt{\varepsilon
_{1}}}\;\forall \;n\in \mathbb{N}_{0},
\end{equation*}%
\begin{equation*}
\left\vert \left( \breve{u}_{M}^{BL}\right) ^{(n)}(x)\right\vert \lesssim 
\breve{K}^{n}\varepsilon _{1}^{-n/2}e^{-dist(x,\partial I)/\sqrt{\varepsilon
_{1}}}\;\forall \;n\in \mathbb{N}_{0},
\end{equation*}%
\begin{equation*}
\left\Vert r_{M}\right\Vert _{\infty ,\partial I}+\left\Vert
r_{M}\right\Vert _{0,I}+\varepsilon _{1}^{1/2}\left\Vert r_{M}^{\prime
}\right\Vert _{0,I}\lesssim ,e^{-\delta /\varepsilon _{2}}
\end{equation*}%
provided $\sqrt{\varepsilon _{1}}K_{2}M<1$.
\end{proposition}

\section{Discretization by an $hp$-FEM\label{mesh}}

\subsection{Discrete formulation and definition of the mesh}

The variational formulation of (\ref{de})--(\ref{bc}) reads: Find $u\in
H_{0}^{1}\left( I\right) $ such that 
\begin{equation}
{\mathcal{B}}\left( u,v\right) ={\mathcal{F}}\left( v\right) \;\;\forall
\;v\in H_{0}^{1}\left( I\right) ,  \label{BuvFv}
\end{equation}%
where 
\begin{eqnarray}
{\mathcal{B}}\left( u,v\right) &=&\varepsilon _{1}\left\langle u^{\prime
},v^{\prime }\right\rangle _{I}+\varepsilon _{2}\left\langle bu^{\prime
},v\right\rangle _{I}+\left\langle cu,v\right\rangle _{I},  \label{Buv} \\
{\mathcal{F}}\left( v\right) &=&\left\langle f,v\right\rangle _{I}.
\label{Fv}
\end{eqnarray}%
The bilinear form ${\mathcal{B}}\left( \cdot ,\cdot \right) $ given by (\ref%
{Buv}) is coercive (due to (\ref{data})) with respect to the \emph{energy
norm} 
\begin{equation}
\left\Vert v\right\Vert _{E,I}^{2}:=\varepsilon _{1}\left\vert v\right\vert
_{1,I}^{2}+\left\Vert v\right\Vert _{0,I}^{2},  \label{energy}
\end{equation}%
i.e., 
\begin{equation}
{\mathcal{B}}\left( v,v\right) \geq \left\Vert v\right\Vert
_{E,I}^{2}\;\;\forall \;v\in H_{0}^{1}\left( I\right) .  \label{coercivity}
\end{equation}

With $S\subset H_{0}^{1}\left( I\right) $ a finite dimensional subspace that
will be defined shortly, the discrete version of (\ref{BuvFv}) reads: find $%
u_{N}\in S$ such that 
\begin{equation}
{\mathcal{B}}\left( u_{N},v\right) ={\mathcal{F}}\left( v\right) \;\;\forall
\;v\in S.  \label{discrete}
\end{equation}%
In order to define the subspace $S$, let $\hat{I}=[-1,1]$ be the reference
element and denote by $\mathcal{P}_{p}(\hat{I})\ $the space of polynomials
on $\hat{I},$ of degree $\leq p$. Then, with $\Delta =\left\{ x_{j}\right\}
_{j=0}^{N}$ an arbitrary subdivision of $I,$ we define 
\begin{equation}
S\equiv S^{p}(\Delta )=\{u\in H_{0}^{1}(\Omega ):u(Q_{j}(\xi ))\in \mathcal{P%
}_{p}(\hat{I}),\,j=1,\dots ,N\},  \label{S}
\end{equation}%
where the linear element mapping is given by $Q_{j}(\xi )=(2\xi
-x_{j-1}-x_{j})/(x_{j}-x_{j-1}).$ \ 

We next give the definition of the \emph{Spectral Boundary Layer Mesh }we
will use (cf. \cite{MXO}):

\begin{definition}[Spectral Boundary Layer mesh]
\label{SBL} Let $\mu _{0},\mu _{1}$ be given by (\ref{mu}). \ For $\kappa >0$%
, $p\in \mathbb{N}$ and $0<\varepsilon _{1},\varepsilon _{2}\leq 1$, define
the Spectral Boundary Layer mesh $\Delta _{BL}(\kappa ,p)$ as%
\begin{equation*}
\Delta _{BL}(\kappa ,p):=\left\{ 
\begin{array}{ll}
\Delta =\{0,1\} & \text{if }\kappa p\mu _{1}^{-1}\geq 1/2 \\ 
\Delta =\{0,\kappa p\mu _{0}^{-1},1-\kappa p\mu _{1}^{-1},1\} & \text{if }%
\kappa p\mu _{0}^{-1}<1/2%
\end{array}%
\right. .
\end{equation*}%
The spaces $S(\kappa ,p)$ and $S_{0}(\kappa ,p)$ of piecewise polynomials of
degree $p$ are given by 
\begin{eqnarray*}
S(\kappa ,p):= &&S^{p}(\Delta _{BL}(\kappa ,p)), \\
S_{0}(\kappa ,p):= &&S_{0}^{p}(\Delta _{BL}(\kappa ,p))=S(\kappa ,p)\cap
H_{0}^{1}(I).
\end{eqnarray*}
\end{definition}

The following tool from \cite{schwab} will be used in the next subsection
for the construction of the approximation.

\begin{proposition}
\label{thm_sch} Let $I=\left( a,b\right) $. \ Then for any $u\in C^{\infty
}(I)$ there exists $\mathcal{I}_{p}u\in \mathcal{P}_{p}(I)$ such that%
\begin{equation}
u\left( a\right) =\mathcal{I}_{p}u\left( a\right) \;,\;u\left( b\right) =%
\mathcal{I}_{p}u\left( b\right) ,  \label{sch_1}
\end{equation}%
\begin{equation}
\left\Vert u-\mathcal{I}_{p}u\right\Vert _{0,I}^{2}\leq \left( b-a\right)
^{2s}\frac{1}{p^{2}}\frac{\left( p-s\right) !}{\left( p+s\right) !}%
\left\Vert u^{(s+1)}\right\Vert _{0,I}^{2}\;,\;0\leq s\leq p,  \label{sch_2}
\end{equation}%
\begin{equation}
\left\Vert \left( u-\mathcal{I}_{p}u\right) ^{\prime }\right\Vert
_{0,I}^{2}\leq \left( b-a\right) ^{2s}\frac{\left( p-s\right) !}{\left(
p+s\right) !}\left\Vert u^{(s+1)}\right\Vert _{0,I}^{2}\;,\;0\leq s\leq p.
\label{sch_3}
\end{equation}
\end{proposition}

The following auxiliary result will be used repeatedly in the proofs that
follow.

\begin{lemma}
\label{stirling} For every $t \in (0,1]$, there exists a constant $C$
(depending on $t \in (0,1]$) such that for every $q\in \mathbb{N}$, there
holds 
\begin{equation*}
\frac{(q-t q)!}{(q+t q)!}\leq C \left[ \frac{\left( 1-t \right) ^{(1-t )}}{%
\left( 1+t \right) ^{(1+t )}}\right] ^{q} q^{-2tq}e^{2tq}.
\end{equation*}
\end{lemma}

\begin{proof} We have $(q\pm t q )!=\Gamma (q\pm t q + 1)$, and as $q\rightarrow \infty $ \cite{AAR},
\begin{eqnarray*}
\frac{(q-t q)!}{(q+t q)!} &=&\frac{\Gamma (q-t q+1)}{\Gamma
(q+t q+1)}\leq C \frac{\left( q\left( 1-t \right) + 1\right)
^{q-t q+1/2}e^{-(q-t q+1)}}{\left( q\left( 1+t \right)+1
\right) ^{q+t q+1/2}e^{-(q+t q+1)}} \\
&\leq& C\left[ \frac{\left( 1-t \right) ^{(1-t )}}{\left(
1+t \right) ^{(1+t )}}\right] ^{q}q^{-2t q}e^{2t q}.
\end{eqnarray*}
\end{proof}

\begin{remark}
In the proofs that follow, we will be using derivatives and norms of
fractional order, as well as non-integer factorials. The corresponding error
estimates may be obtained by either classical interpolation arguments or by
the log-convexity of the $\Gamma $ function.
\end{remark}

\subsection{Error estimates}

\label{approx}

We begin with the following lemma, which provides an estimate for the
interpolation error.

\begin{lemma}
\label{lemma_approx}Let $u$ be the solution of (\ref{de}), (\ref{bc}) and
let $\mathcal{I}_{p}$ be the approximation operator of Theorem \ref{thm_sch}%
. \ Then there exists a constant $\sigma >0$, independent of $\varepsilon
_{1},\varepsilon _{2}$, such that%
\begin{equation*}
\left\Vert u-\mathcal{I}_{p}u\right\Vert _{E,I}\lesssim e^{-\sigma p}.
\end{equation*}
\end{lemma}

\begin{proof} The proof is separated into two cases:

\textit{Case 1}: $\kappa p\mu _{1}^{-1}\geq 1/2$ (asymptotic case)

In this case the mesh consists of only one element and by Theorem \ref%
{thm_reg0}, there holds%
\begin{equation*}
\left\Vert u^{(n)}\right\Vert _{0,I}\lesssim K^{n}\max \left\{ n,\varepsilon
_{1}^{-1},\varepsilon _{2}^{-1}\right\} ^{n}=K^{n}\max \left\{ n,\varepsilon
_{1}^{-1}\right\} ^{n},
\end{equation*}%
since we assumed $\varepsilon _{1}<\varepsilon _{2}$. By Theorem \ref%
{thm_sch}, there exists $\mathcal{I}_{p}u\in \mathcal{P}_{p}(I)$ such that%
\begin{equation*}
\left\Vert u-\mathcal{I}_{p}u\right\Vert _{0,I}^{2}\lesssim \frac{1}{p^{2}}%
\frac{\left( p-s\right) !}{\left( p+s\right) !}\left\Vert
u^{(s+1)}\right\Vert _{0,I}^{2}\;,\;0\leq s\leq p.
\end{equation*}%
Choose $s=\lambda p$, with $\lambda \in (0,1)$ to be chosen shortly. Then%
\begin{equation*}
\left\Vert u-\mathcal{I}_{p}u\right\Vert _{0,I}^{2}\lesssim \frac{1}{p^{2}}%
\frac{\left( p-s\right) !}{\left( p+s\right) !}K^{2(\lambda p+1)}\max
\left\{ \lambda p+1,\varepsilon _{1}^{-1}\right\} ^{2(\lambda p+1)}
\end{equation*}%
and since $\kappa p\mu _{1}^{-1}\geq 1/2,$ we have $\kappa p\varepsilon
_{1}\geq 1/2$ by (\ref{mu_a}), and as a result 
\begin{equation*}
\max \left\{ \lambda p+1,\varepsilon _{1}^{-1}\right\} ^{2(\lambda
p+1)}=(\lambda p+1)^{2(\lambda p+1)},
\end{equation*}%
provided the constant $\kappa $ satisfies $\kappa \leq \lambda /2\,$. Lemma
\ref{stirling} gives
\begin{eqnarray*}
\left\Vert u-\mathcal{I}_{p}u\right\Vert _{0,I}^{2} &\lesssim &\frac{1}{p^{2}%
}\frac{\left( p-\lambda p\right) !}{\left( p+\lambda p\right) !}K^{2(\lambda
p+1)}(\lambda p+1)^{2(\lambda p+1)} \\
&\lesssim &\frac{1}{p^{2}}\left[ \frac{(1-\lambda )^{(1-\lambda )}}{%
(1+\lambda )^{(1+\lambda )}}\right] ^{p}e^{2\lambda p+1}K^{2(\lambda
p+1)}(\lambda p+1)^{2}\left( \frac{\lambda p+1}{p}\right) ^{2\lambda p} \\
&\lesssim &eK^{2}\left[ \frac{(1-\lambda )^{(1-\lambda )}}{(1+\lambda
)^{(1+\lambda )}}\left( eK\right) ^{2\lambda }\right] ^{p}\left( \frac{1}{p}%
+\lambda \right) ^{2\lambda p}.
\end{eqnarray*}%
Since $\left( \frac{1}{p}+\lambda \right) ^{2\lambda p}=\lambda ^{2\lambda p}%
\left[ \left( 1+\frac{1}{\lambda p}\right) ^{\lambda p}\right] ^{2}\leq
e^{2}\lambda ^{2\lambda p}$, we futher have%
\begin{equation*}
\left\Vert u-\mathcal{I}_{p}u\right\Vert _{0,I}^{2}\lesssim \left[ \frac{%
(1-\lambda )^{(1-\lambda )}}{(1+\lambda )^{(1+\lambda )}}\left( eK\lambda
\right) ^{2\lambda }\right] ^{p}
\end{equation*}%
If we choose $\lambda =(eK)^{-1}\in (0,1)$ then we obtain%
\begin{equation*}
\left\Vert u-\mathcal{I}_{p}u\right\Vert _{0,I}^{2}\lesssim \left[ \frac{%
(1-\lambda )^{(1-\lambda )}}{(1+\lambda )^{(1+\lambda )}}\right]
^{p}\lesssim e^{-\beta _{1}p},
\end{equation*}%
where%
\begin{equation*}
\beta _{1}=\left\vert \ln q_{1}\right\vert ,q_{1}=\frac{(1-\lambda
)^{(1-\lambda )}}{(1+\lambda )^{(1+\lambda )}}<1.
\end{equation*}%
We note that the choice of $\lambda $ implies that the constant $\kappa $ in
the definition of the mesh, satisfies $\kappa <\frac{1}{2eK}.$

Following the same steps as above and using Theorems \ref{thm_reg0} and \ref%
{thm_sch}, we may show 
\begin{equation*}
\left\Vert \left( u-\mathcal{I}_{p}u\right) ^{\prime }\right\Vert
_{0,I}^{2}\lesssim p^{2}e^{-\beta _{1}p},
\end{equation*}%
so that combining the two, gives the desired result (note that the $p^{2}$
term above may be absorbed into the exponential by adjusting the constants).

\textit{Case 2}: $\kappa p\mu _{0}^{-1}<1/2$ (pre-asymptotic case)

In this case the mesh is given by 
\begin{equation*}
\Delta _{BL}(\kappa ,p):=\{0,\kappa p\mu _{0}^{-1},1-\kappa p\mu
_{1}^{-1},1\},
\end{equation*}%
and the solution is decomposed based on the relationship between $%
\varepsilon _{1}$ and $\varepsilon _{2}.$ We will consider the first regime
(see Section \ref{regime1}) and note that the approximation for the other
two regimes (see Sections \ref{regime2} and \ref{regime3}) is analogous (see
also \cite{melenk}).

So we assume $\varepsilon _{1}<<\varepsilon _{2}^{2}<<1$ and we have the
decomposition (\ref{decomp1}):%
\begin{equation*}
u=u_{M}+\tilde{u}_{M}^{BL}+\hat{u}_{M}^{BL}+r_{M},
\end{equation*}%
with each term satisfying the bounds presented in Theorem \ref{thm_reg1}. We
will construct a different approximation for each part, using Theorem \ref%
{thm_sch}.

For the smooth part $u_{M},$ we have that there exists $\mathcal{I}%
_{p}u_{M}\in \mathcal{P}_{p}(I)$ such that 
\begin{equation*}
\left\Vert u_{M}-\mathcal{I}_{p}u_{M}\right\Vert _{0,I}^{2}+\left\Vert
\left( u_{M}-\mathcal{I}_{p}u_{M}\right) ^{\prime }\right\Vert
_{0,I}^{2}\lesssim \frac{\left( p-s\right) !}{\left( p+s\right) !}\left\Vert
u_{M}^{(s+1)}\right\Vert _{0,I}^{2}\;,\;0\leq s\leq p.
\end{equation*}%
Choose $s=\bar{\lambda}p$, with $\bar{\lambda}\in (0,1)$ to be chosen
shortly. Then, utilizing the estimate (\ref{uM1bound}) and Lemma \ref{stirling}, 
we arrive at
\begin{equation*}
\left\Vert u_{M}-\mathcal{I}_{p}u_{M}\right\Vert _{0,I}^{2}+\left\Vert
\left( u_{M}-\mathcal{I}_{p}u_{M}\right) ^{\prime }\right\Vert
_{0,I}^{2}\lesssim \frac{\left( p-s\right) !}{\left( p+s\right) !}K_{1}^{2(%
\bar{\lambda}p+1)}\left( \bar{\lambda}p+1\right) ^{2(\bar{\lambda}p+1)}
\end{equation*}%
\begin{eqnarray*}
&\lesssim &\left[ \frac{(1-\bar{\lambda})^{(1-\bar{\lambda})}}{(1+\bar{%
\lambda})^{(1+\bar{\lambda})}}\right] ^{p}e^{2\bar{\lambda}p+1}K_{1}^{2(\bar{%
\lambda}p+1)}(\bar{\lambda}p+1)^{2}\left( \frac{\bar{\lambda}p+1}{p}\right)
^{2\bar{\lambda}p} \\
&\lesssim &p^{2}eK_{1}^{2}\left[ \frac{(1-\bar{\lambda})^{(1-\bar{\lambda})}%
}{(1+\bar{\lambda})^{(1+\bar{\lambda})}}\left( eK_{1}\right) ^{2\bar{\lambda}%
}\right] ^{p}\left( \frac{1}{p}+\bar{\lambda}\right) ^{2\bar{\lambda}p} \\
&\lesssim &p^{2}eK_{1}^{2}\left[ \frac{(1-\bar{\lambda})^{(1-\bar{\lambda})}%
}{(1+\bar{\lambda})^{(1+\bar{\lambda})}}\left( eK_{1}\bar{\lambda}\right) ^{2%
\bar{\lambda}}\right] ^{p}.
\end{eqnarray*}%
Following the same reasoning as in \textit{Case 1} above, i.e. choosing 
$\bar{\lambda}=(eK_{1})^{-1}$ etc., we obtain 
\begin{equation*}
\left\Vert u_{M}-\mathcal{I}_{p}u_{M}\right\Vert _{E,I}\lesssim pe^{-\sigma
p}.
\end{equation*}

For the left boundary layer $\tilde{u}_{M}^{BL}$, we will construct
different approximations on the intervals%
\begin{equation*}
\tilde{I}_{1}=[0,\kappa p\mu _{0}^{-1}]\,,\,\tilde{I}_{2}=[\kappa p\mu
_{0}^{-1},1].
\end{equation*}%
On $\tilde{I}_{1},$ Theorem \ref{thm_sch} gives the existence of $\mathcal{I}%
_{p}\tilde{u}_{M}^{BL}\in \mathcal{P}_{p}(\tilde{I}_{1})$ such that%
\begin{equation*}
\left\Vert \left( \tilde{u}_{M}^{BL}-\mathcal{I}_{p}\tilde{u}%
_{M}^{BL}\right) ^{\prime }\right\Vert _{0,\tilde{I}_{1}}^{2}\lesssim \left(
\kappa p\mu _{0}^{-1}\right) ^{2s}\frac{\left( p-s\right) !}{\left(
p+s\right) !}\left\Vert \left( \tilde{u}_{M}^{BL}\right) ^{(s+1)}\right\Vert
_{0,\tilde{I}_{1}}^{2}\;,\;0\leq s\leq p.
\end{equation*}%
Choose $s=\tilde{\lambda}p$, with $\tilde{\lambda}\in (0,1)$ arbitrary.
Then, with the aid of Lemma \ref{stirling}, we have%
\begin{eqnarray*}
\left\Vert \left( \tilde{u}_{M}^{BL}-\mathcal{I}_{p}\tilde{u}%
_{M}^{BL}\right) ^{\prime }\right\Vert _{0,\tilde{I}_{1}}^{2} &\lesssim
&\left( \kappa p\mu _{0}^{-1}\right) ^{2\tilde{\lambda}p}\frac{\left( p-%
\tilde{\lambda}p\right) !}{\left( p+\tilde{\lambda}p\right) !}\left\Vert
\left( \tilde{u}_{M}^{BL}\right) ^{(\tilde{\lambda}p+1)}\right\Vert _{0,%
\tilde{I}_{1}}^{2} \\
&\lesssim &\left( \kappa p\mu _{0}^{-1}\right) ^{2\tilde{\lambda}p}\left[ 
\frac{(1-\tilde{\lambda})^{(1-\tilde{\lambda})}}{(1+\tilde{\lambda})^{(1+%
\tilde{\lambda})}}\right] ^{p}p^{-2\tilde{\lambda}p}e^{2\tilde{\lambda}%
p+1}\left\Vert \left( \tilde{u}_{M}^{BL}\right) ^{(\tilde{\lambda}%
p+1)}\right\Vert _{0,\tilde{I}_{1}}^{2}.
\end{eqnarray*}%
By (\ref{uBL1abound}),%
\begin{eqnarray*}
\left\Vert \left( \tilde{u}_{M}^{BL}\right) ^{(\tilde{\lambda}%
p+1)}\right\Vert _{0,\tilde{I}_{1}}^{2} &=&\int_{0}^{\kappa p\mu
_{0}^{-1}}\left( \tilde{u}_{M}^{BL}\right) ^{2(\tilde{\lambda}p+1)}(x)dx \\
&\lesssim &\int_{0}^{\kappa p\mu _{0}^{-1}}\tilde{K}^{2(\tilde{\lambda}%
p+1)}\varepsilon _{2}^{-2(\tilde{\lambda}p+1)}e^{-2dist(x,\partial
I)/\varepsilon _{2}}dx \\
&\lesssim &\kappa p\mu _{0}^{-1}\tilde{K}^{2(\tilde{\lambda}p+1)}\varepsilon
_{2}^{-2(\tilde{\lambda}p+1)},
\end{eqnarray*}%
so that%
\begin{eqnarray*}
\left\Vert \left( \tilde{u}_{M}^{BL}-\mathcal{I}_{p}\tilde{u}%
_{M}^{BL}\right) ^{\prime }\right\Vert _{0,\tilde{I}_{1}}^{2} &\lesssim
&\left( \kappa p\mu _{0}^{-1}\right) ^{2\tilde{\lambda}p}\left[ \frac{(1-%
\tilde{\lambda})^{(1-\tilde{\lambda})}}{(1+\tilde{\lambda})^{(1+\tilde{%
\lambda})}}\right] ^{p}p^{-2\tilde{\lambda}p}e^{2\tilde{\lambda}p+1}\kappa
p\mu _{0}^{-1}\tilde{K}^{2(\tilde{\lambda}p+1)}\varepsilon _{2}^{-2(\tilde{%
\lambda}p+1)} \\
&\lesssim &p\left( \mu _{0}^{-1}\right) ^{2\tilde{\lambda}p+1}\varepsilon
_{2}^{-2(\tilde{\lambda}p+1)}\left[ \frac{(1-\tilde{\lambda})^{(1-\tilde{%
\lambda})}}{(1+\tilde{\lambda})^{(1+\tilde{\lambda})}}\right] ^{p}\left(
\kappa e\tilde{K}\right) ^{2\tilde{\lambda}p} \\
&\lesssim &p\varepsilon _{2}^{-1}\left( \mu _{0}\varepsilon _{2}\right) ^{-(2%
\tilde{\lambda}p+1)}\left[ \frac{(1-\tilde{\lambda})^{(1-\tilde{\lambda})}}{%
(1+\tilde{\lambda})^{(1+\tilde{\lambda})}}\right] ^{p},
\end{eqnarray*}%
by the choice of $\kappa <1/(e\tilde{K}).$ Since by (\ref{mu_a}) in this
regime there holds $\mu _{0}\varepsilon _{2}\lesssim 1$, we get%
\begin{equation*}
\left\Vert \left( \tilde{u}_{M}^{BL}-\mathcal{I}_{p}\tilde{u}%
_{M}^{BL}\right) ^{\prime }\right\Vert _{0,\tilde{I}_{1}}\lesssim \sqrt{p}%
\varepsilon _{2}^{-1/2}e^{-\beta _{2}p},
\end{equation*}%
with 
\begin{equation*}
\beta _{2}=\left\vert \ln q_{2}\right\vert ,q_{2}=\frac{(1-\tilde{\lambda}%
)^{(1-\tilde{\lambda})}}{(1+\tilde{\lambda})^{(1+\tilde{\lambda})}}<1.
\end{equation*}%
On the interval $\tilde{I}_{2}=[\kappa p\mu _{0}^{-1},1]$, we approximate $%
\tilde{u}_{M}^{BL}$ by its linear interpolant $\mathcal{I}_{1}\tilde{u}%
_{M}^{BL}$ and we have%
\begin{eqnarray*}
\left\Vert \left( \tilde{u}_{M}^{BL}-\mathcal{I}_{1}\tilde{u}%
_{M}^{BL}\right) ^{\prime }\right\Vert _{0,\tilde{I}_{2}}^{2} &\lesssim
&\left\Vert \left( \tilde{u}_{M}^{BL}\right) ^{\prime }\right\Vert _{0,%
\tilde{I}_{2}}^{2}+\left\Vert \left( \mathcal{I}_{1}\tilde{u}%
_{M}^{BL}\right) ^{\prime }\right\Vert _{0,\tilde{I}_{2}}^{2} \\
&\lesssim &\int_{\kappa p\mu _{0}^{-1}}^{1}\left[ \left( \tilde{u}%
_{M}^{BL}\right) ^{\prime }(x)\right] ^{2}dx\lesssim \int_{\kappa p\mu
_{0}^{-1}}^{1}\varepsilon _{2}^{-2}e^{-2dist(x,\partial I)/\varepsilon
_{2}}dx \\
&\lesssim &\varepsilon _{2}^{-1}e^{-2\kappa p\mu _{0}^{-1}/\varepsilon _{2}}
\\
&\lesssim &\varepsilon _{2}^{-1}e^{-2\kappa p},
\end{eqnarray*}%
by (\ref{mu_a}). Therefore,%
\begin{equation*}
\left\Vert \left( \tilde{u}_{M}^{BL}-\mathcal{I}_{p}\tilde{u}%
_{M}^{BL}\right) ^{\prime }\right\Vert _{0,I}\lesssim \varepsilon
_{2}^{-1/2}e^{-\sigma p},
\end{equation*}%
for some $\sigma >0$, independent of $\varepsilon _{1},\varepsilon _{2}.$
Repeating the argument for the $L^{2}$ norm of the error and using the
definition of the energy norm, we get%
\begin{eqnarray*}
\left\Vert \left( \tilde{u}_{M}^{BL}-\mathcal{I}_{p}\tilde{u}%
_{M}^{BL}\right) \right\Vert _{E,I} &\lesssim &\varepsilon
_{1}^{1/2}\left\Vert \left( \tilde{u}_{M}^{BL}-\mathcal{I}_{p}\tilde{u}%
_{M}^{BL}\right) ^{\prime }\right\Vert _{0,I}+\left\Vert \left( \tilde{u}%
_{M}^{BL}-\mathcal{I}_{p}\tilde{u}_{M}^{BL}\right) \right\Vert _{0,I} \\
&\lesssim &\varepsilon _{1}^{1/2}\varepsilon _{2}^{-1/2}e^{-\sigma
p}+e^{-\sigma p} \\
&\lesssim &e^{-\sigma p},
\end{eqnarray*}%
since $\varepsilon _{1}^{1/2}\varepsilon _{2}^{-1/2}=O(1)$ due to $%
\varepsilon _{1}<\varepsilon _{2}.$

For the right boundary layer $\hat{u}_{M}^{BL}$, we will construct different
approximations on the intervals%
\begin{equation*}
\hat{I}_{1}=[0,1-\kappa p\mu _{1}^{-1}]\,,\,\hat{I}_{2}=[1-\kappa p\mu
_{1}^{-1},1].
\end{equation*}%
The steps are the same as for the left boundary layer. On $\hat{I}_{1}$ we
use the linear interpolant $\mathcal{I}_{1}\hat{u}_{M}^{BL}$ for the
approximation, getting with the help of (\ref{uBL1bbound}),%
\begin{eqnarray*}
\left\Vert \left( \hat{u}_{M}^{BL}-\mathcal{I}_{1}\hat{u}_{M}^{BL}\right)
^{\prime }\right\Vert _{0,\hat{I}_{1}}^{2} &\lesssim &\left\Vert \left( \hat{%
u}_{M}^{BL}\right) ^{\prime }\right\Vert _{0,\hat{I}_{1}}^{2}+\left\Vert
\left( \mathcal{I}_{1}\hat{u}_{M}^{BL}\right) ^{\prime }\right\Vert _{0,\hat{%
I}_{1}}^{2}\lesssim \int_{0}^{1-\kappa p\mu _{1}^{-1}}\left[ \left( \hat{u}%
_{M}^{BL}\right) ^{\prime }\right] ^{2} \\
&\lesssim &\int_{0}^{1-\kappa p\mu _{1}^{-1}}K_{2}^{2}\left( \frac{%
\varepsilon _{1}}{\varepsilon _{2}}\right) ^{-2}e^{-2dist(x,\partial
I)\varepsilon _{2}/\varepsilon _{1}}dx \\
&\lesssim &\kappa p\mu _{1}^{-1}\left( \frac{\varepsilon _{1}}{\varepsilon
_{2}}\right) ^{-2}e^{-2\kappa p\mu _{1}^{-1}\varepsilon _{2}/\varepsilon
_{1}} \\
&\lesssim &pe^{-2\kappa p},
\end{eqnarray*}%
where we used (\ref{mu_a}). On $\hat{I}_{2}$, we have by Theorem \ref%
{thm_sch} that there exists $\mathcal{I}_{p}\hat{u}_{M}^{BL}\in \mathcal{P}%
_{p}(\hat{I}_{2})$ such that 
\begin{equation*}
\left\Vert \left( \hat{u}_{M}^{BL}-\mathcal{I}_{p}\hat{u}_{M}^{BL}\right)
^{\prime }\right\Vert _{0,\hat{I}_{2}}^{2}\lesssim \left( \kappa p\mu
_{1}^{-1}\right) ^{2s}\frac{\left( p-s\right) !}{\left( p+s\right) !}%
\left\Vert \left( \hat{u}_{M}^{BL}\right) ^{(s+1)}\right\Vert _{0,\hat{I}%
_{2}}^{2}\;,\;0\leq s\leq p.
\end{equation*}%
Choose $s=\hat{\lambda}p$, with $\hat{\lambda}\in (0,1)$ arbitrary. Then,
with the aid of Lemma \ref{stirling}, we have%
\begin{eqnarray*}
\left\Vert \left( \hat{u}_{M}^{BL}-\mathcal{I}_{p}\hat{u}_{M}^{BL}\right)
^{\prime }\right\Vert _{0,\hat{I}_{2}}^{2} &\lesssim &\left( \kappa p\mu
_{1}^{-1}\right) ^{2\hat{\lambda}p}\frac{\left( p-\hat{\lambda}p\right) !}{%
\left( p+\hat{\lambda}p\right) !}\left\Vert \left( \hat{u}_{M}^{BL}\right)
^{(\hat{\lambda}p+1)}\right\Vert _{0,\hat{I}_{2}}^{2} \\
&\lesssim &\left( \kappa p\mu _{1}^{-1}\right) ^{2\hat{\lambda}p}\left[ 
\frac{(1-\hat{\lambda})^{(1-\hat{\lambda})}}{(1+\hat{\lambda})^{(1+\hat{%
\lambda})}}\right] ^{p}p^{-2\hat{\lambda}p}e^{2\hat{\lambda}p+1}\left\Vert
\left( \hat{u}_{M}^{BL}\right) ^{(\hat{\lambda}p+1)}\right\Vert _{0,\hat{I}%
_{2}}^{2}.
\end{eqnarray*}%
By (\ref{uBL1bbound}),%
\begin{eqnarray*}
\left\Vert \left( \hat{u}_{M}^{BL}\right) ^{(\hat{\lambda}p+1)}\right\Vert
_{0,\hat{I}_{2}}^{2} &=&\int_{1-\kappa p\mu _{1}^{-1}}^{1}\left( \hat{u}%
_{M}^{BL}\right) ^{2(\hat{\lambda}p+1)} \\
&\lesssim &\int_{1-\kappa p\mu _{1}^{-1}}^{1}\hat{K}^{2(\hat{\lambda}%
p+1)}\left( \frac{\varepsilon _{1}}{\varepsilon _{2}}\right) ^{-2(\hat{%
\lambda}p+1)}e^{-2dist(x,\partial I)\varepsilon _{2}/\varepsilon _{1}}dx \\
&\lesssim &\hat{K}^{2(\hat{\lambda}p+1)}\left( \frac{\varepsilon _{1}}{%
\varepsilon _{2}}\right) ^{-2\hat{\lambda}p+1}
\end{eqnarray*}%
so that%
\begin{eqnarray*}
\left\Vert \left( \hat{u}_{M}^{BL}-\mathcal{I}_{p}\hat{u}_{M}^{BL}\right)
^{\prime }\right\Vert _{0,\hat{I}_{2}}^{2} &\lesssim &\left( \kappa p\mu
_{1}^{-1}\right) ^{2\hat{\lambda}p}\left[ \frac{(1-\hat{\lambda})^{(1-\hat{%
\lambda})}}{(1+\hat{\lambda})^{(1+\hat{\lambda})}}\right] ^{p}p^{-2\hat{%
\lambda}p}e^{2\hat{\lambda}p+1}\hat{K}^{2(\hat{\lambda}p+1)}\left( \frac{%
\varepsilon _{1}}{\varepsilon _{2}}\right) ^{-2\hat{\lambda}p+1} \\
&\lesssim &\left( \mu _{1}^{-1}\right) ^{2\hat{\lambda}p}\left( \frac{%
\varepsilon _{1}}{\varepsilon _{2}}\right) ^{-2\hat{\lambda}p+1}\left[ \frac{%
(1-\hat{\lambda})^{(1-\hat{\lambda})}}{(1+\hat{\lambda})^{(1+\hat{\lambda})}}%
\right] ^{p}\left( \kappa e\hat{K}\right) ^{2\hat{\lambda}p} \\
&\lesssim &\frac{\varepsilon _{1}}{\varepsilon _{2}}\left( \mu _{1}^{-1}%
\frac{\varepsilon _{2}}{\varepsilon _{1}}\right) ^{2\hat{\lambda}p}\left[ 
\frac{(1-\hat{\lambda})^{(1-\hat{\lambda})}}{(1+\hat{\lambda})^{(1+\hat{%
\lambda})}}\right] ^{p}.
\end{eqnarray*}%
Since in this regime there holds $\mu _{1}^{-1}\frac{\varepsilon _{2}}{%
\varepsilon _{1}}\lesssim 1$ by (\ref{mu_a}), we get%
\begin{equation}
\left\Vert \left( \hat{u}_{M}^{BL}-\mathcal{I}_{p}\hat{u}_{M}^{BL}\right)
^{\prime }\right\Vert _{0,\hat{I}_{2}}^{2}\lesssim \frac{\varepsilon _{1}}{%
\varepsilon _{2}}e^{-\beta _{3}p},  \label{LeqNo}
\end{equation}%
with 
\begin{equation*}
\beta _{3}=\left\vert \ln q_{3}\right\vert ,q_{3}=\frac{(1-\hat{\lambda}%
)^{(1-\hat{\lambda})}}{(1+\hat{\lambda})^{(1+\hat{\lambda})}}<1.
\end{equation*}
For the $L^{2}$ error, we have in an analogous fashion%
\begin{equation*}
\left\Vert \hat{u}_{M}^{BL}-\mathcal{I}_{p}\hat{u}_{M}^{BL}\right\Vert _{0,%
\hat{I}_{2}}\lesssim e^{-\sigma p},
\end{equation*}
so that the above considerations yield
\begin{equation*}
\left\Vert \hat{u}_{M}^{BL}-\mathcal{I}_{p}\hat{u}_{M}^{BL}\right\Vert
_{E,I}\lesssim \left( \varepsilon_1 (\varepsilon_1^{1/2} / \varepsilon_2^{1/2}) +1 \right) e^{-\sigma p} \lesssim e^{-\sigma p},
\end{equation*}%
with $\sigma >0$ a constant independent of $\varepsilon _{1},\varepsilon _{2}
$.

We finally consider the remainder, $r_{M}$ which satisfies (\ref{rM1bound}),
or equivalently%
\begin{equation*}
\left\Vert r_{M}\right\Vert _{E,I}\lesssim \max \{e^{-\delta \varepsilon
_{2}/\varepsilon _{1}},e^{-\delta /\varepsilon _{2}}\}\lesssim e^{-\delta
/\varepsilon _{2}},
\end{equation*}%
due to $\varepsilon _{2}^{2}>>\varepsilon _{1}.$ Since the remainder is
already exponentially small, it will not be approximated and we note that $%
\kappa p\mu _{0}^{-1}<1/2$ implies $\kappa p\varepsilon _{2}<1/2$, hence%
\begin{equation*}
\left\Vert r_{M}\right\Vert _{E,I}\lesssim e^{-\sigma \kappa p},
\end{equation*}%
with $\sigma >0$ a constant independent of $\varepsilon _{1},\varepsilon
_{2} $. Combining all the above we obtain the desired result. \end{proof}

We next estimate the error between the finite element solution $u_{FEM}$ and
the interpolant $\mathcal{I}_{p}u$.

\begin{lemma}
\label{lemma_approx2}Let $u$ be the solution of (\ref{de})--(\ref{bc}), $%
u_{FEM}\in S_{0}(\kappa ,p)$ be its approximation based on the Spectral
Boundary Layer Mesh, and let $\mathcal{I}_{p}$ be the approximation operator
of Theorem \ref{thm_sch}. Then there exists a constant $\sigma >0$,
independent of $\varepsilon _{1},\varepsilon _{2}$, such that%
\begin{equation*}
\left\Vert \mathcal{I}_{p}u-u_{FEM}\right\Vert _{E,I}\lesssim e^{-\sigma p}.
\end{equation*}
\end{lemma}

\begin{proof} By coercivity of the bilinear form ${\mathcal{B}}_{\varepsilon }
$ (eq. (\ref{coercivity})), there holds with $\xi :=\mathcal{I}_{p}u-u_{FEM},
$%
\begin{equation*}
\left\Vert \xi \right\Vert _{E,I}^{2}\leq {\mathcal{B}}_{\varepsilon }\left(
\xi ,\xi \right) =-{\mathcal{B}}_{\varepsilon }\left( u-\mathcal{I}_{p}u,\xi
\right) ,
\end{equation*}%
where we also used Galerkin orthogonality. Hence%
\begin{equation*}
\left\Vert \xi \right\Vert _{E,I}^{2}\leq -\varepsilon _{1}\left\langle
\left( u-\mathcal{I}_{p}u\right) ^{\prime },\xi ^{\prime }\right\rangle
_{I}-\varepsilon _{2}\left\langle b\left( u-\mathcal{I}_{p}u\right) ^{\prime
},\xi \right\rangle _{I}-\left\langle c\left( u-\mathcal{I}_{p}u\right) ,\xi
\right\rangle _{I}.
\end{equation*}%
The first and last term may be estimated using Cauchy Schwarz:%
\begin{equation*}
\left\vert -\varepsilon _{1}\left\langle \left( u-\mathcal{I}_{p}u\right)
^{\prime },\xi ^{\prime }\right\rangle _{I}\right\vert +\left\vert
\left\langle c\left( u-\mathcal{I}_{p}u\right) ,\xi \right\rangle
_{I}\right\vert \lesssim \varepsilon _{1}\left\Vert \left( u-\mathcal{I}%
_{p}u\right) ^{\prime }\right\Vert _{0,I}\left\Vert \xi ^{\prime
}\right\Vert _{0,I}+
\end{equation*}%
\begin{equation*}
+\left\Vert c\right\Vert _{\infty ,I}\left\Vert u-\mathcal{I}%
_{p}u\right\Vert _{0,I}\left\Vert \xi \right\Vert _{0,I}\lesssim \max
\{1,\left\Vert c\right\Vert _{\infty ,I}\}\left\Vert u-\mathcal{I}%
_{p}u\right\Vert _{E,I}\left\Vert \xi \right\Vert _{E,I}.
\end{equation*}%
For the second term, we will consider the two ranges of $p$ separately: in
the asymptotic range of $p$, i.e. $\kappa p\mu _{1}^{-1}\geq 1/2$ or
equivalently $\kappa p\varepsilon _{1}\geq 1/2$, we have%
\begin{eqnarray*}
\left\vert \varepsilon _{2}\left\langle b\left( u-\mathcal{I}_{p}u\right)
^{\prime },\xi \right\rangle _{I}\right\vert  &\lesssim &\varepsilon
_{2}\left\Vert b\right\Vert _{\infty ,I}\left\Vert \left( u-\mathcal{I}%
_{p}u\right) ^{\prime }\right\Vert _{0,I}\left\Vert \xi \right\Vert _{0,I} \\
&\lesssim &\varepsilon _{2}\varepsilon _{1}^{-1/2}\left\Vert u-\mathcal{I}%
_{p}u\right\Vert _{E,I}\left\Vert \xi \right\Vert _{E,I} \\
&\lesssim &\varepsilon _{2}\left( \kappa p\right) ^{1/2}\left\Vert u-%
\mathcal{I}_{p}u\right\Vert _{E,I}\left\Vert \xi \right\Vert _{E,I} \\
&\lesssim &e^{-\sigma p}\left\Vert \xi \right\Vert _{E,I}
\end{eqnarray*}%
In the pre-asymptotic range of $p$, i.e. $\kappa p\mu _{0}^{-1}<1/2$, we
first use integration by parts to obtain%
\begin{equation*}
\left\vert \varepsilon _{2}\left\langle b\left( u-\mathcal{I}_{p}u\right)
^{\prime },\xi \right\rangle _{I}\right\vert =\left\vert \varepsilon
_{2}\left\langle b\left( u-\mathcal{I}_{p}u\right) ,\xi ^{\prime
}\right\rangle _{I}\right\vert .
\end{equation*}%
Next, we consider the three intervals of the \emph{Spectral Boundary Layer}
mesh 
\begin{equation*}
\lbrack 0,\kappa p\mu _{0}^{-1}]\cup \lbrack \kappa p\mu _{0}^{-1},1-\kappa
p\mu _{1}^{-1}]\cup \lbrack 1-\kappa p\mu _{1}^{-1},1].
\end{equation*}%
On the first subinterval we have 
\begin{eqnarray*}
\left\vert \varepsilon _{2}\left\langle b\left( u-\mathcal{I}_{p}u\right)
,\xi ^{\prime }\right\rangle _{[0,\kappa p\mu _{0}^{-1}]}\right\vert 
&\lesssim &\varepsilon _{2}\left\Vert b\right\Vert _{\infty ,[0,\kappa p\mu
_{0}^{-1}]}\left\vert \left\langle u-\mathcal{I}_{p}u,\xi ^{\prime
}\right\rangle _{[0,\kappa p\mu _{0}^{-1}]}\right\vert  \\
&\lesssim &\varepsilon _{2}\left\Vert u-\mathcal{I}_{p}u\right\Vert
_{0,[0,\kappa p\mu _{0}^{-1}]}\left\Vert \xi ^{\prime }\right\Vert
_{0,[0,\kappa p\mu _{0}^{-1}]} \\
&\lesssim &\frac{\varepsilon _{2}}{\kappa p\mu _{0}^{-1}}\left\Vert u-%
\mathcal{I}_{p}u\right\Vert _{0,[0,\kappa p\mu _{0}^{-1}]}\left\Vert \xi
\right\Vert _{0,[0,\kappa p\mu _{0}^{-1}]} \\
&\lesssim &\frac{\varepsilon _{2}\mu _{0}}{\kappa p}\left\Vert u-\mathcal{I}%
_{p}u\right\Vert _{0,[0,\kappa p\mu _{0}^{-1}]}\left\Vert \xi \right\Vert
_{0,[0,\kappa p\mu _{0}^{-1}]}
\end{eqnarray*}%
where we used an inverse inequality (see, e.g. \cite[Thm. 3.91]{schwab}). \
Thus, (\ref{mu_a}) and Lemma \ref{lemma_approx} give%
\begin{equation*}
\left\vert \varepsilon _{2}\left\langle b\left( u-\mathcal{I}_{p}u\right)
,\xi ^{\prime }\right\rangle _{[0,\kappa p\mu _{0}^{-1}]}\right\vert
\lesssim e^{-\beta p}\left\Vert \xi \right\Vert _{E,I}.
\end{equation*}%
Similarly, on the second subinterval we have%
\begin{eqnarray*}
\left\vert \varepsilon _{2}\left\langle b\left( u-\mathcal{I}_{p}u\right)
,\xi ^{\prime }\right\rangle _{[\kappa p\mu _{0}^{-1},1-\kappa p\mu
_{1}^{-1}]}\right\vert  &\lesssim &\varepsilon _{2}\left\Vert b\right\Vert
_{\infty ,[\kappa p\mu _{0}^{-1},1-\kappa p\mu _{1}^{-1}]}\left\vert
\left\langle u-\mathcal{I}_{p}u,\xi ^{\prime }\right\rangle _{[\kappa p\mu
_{0}^{-1},1-\kappa p\mu _{1}^{-1}]}\right\vert  \\
&\lesssim &\varepsilon _{2}\left\Vert u-\mathcal{I}_{p}u\right\Vert
_{0,[\kappa p\mu _{0}^{-1},1-\kappa p\mu _{1}^{-1}]}\left\Vert \xi
\right\Vert _{0,[\kappa p\mu _{0}^{-1},1-\kappa p\mu _{1}^{-1}]} \\
&\lesssim &e^{-\beta p}\left\Vert \xi \right\Vert _{E,I}.
\end{eqnarray*}%
Finally, on the third subinterval%
\begin{eqnarray*}
\left\vert \varepsilon _{2}\left\langle b\left( u-\mathcal{I}_{p}u\right)
,\xi ^{\prime }\right\rangle _{[1-\kappa p\mu _{1}^{-1},1]}\right\vert 
&\lesssim &\varepsilon _{2}\left\Vert b\right\Vert _{\infty ,[1-\kappa p\mu
_{1}^{-1},1]}\left\vert \left\langle u-\mathcal{I}_{p}u,\xi ^{\prime
}\right\rangle _{[1-\kappa p\mu _{1}^{-1},1]}\right\vert  \\
&\lesssim &\frac{\varepsilon _{2}}{\varepsilon _{1}^{1/2}}\left\Vert u-%
\mathcal{I}_{p}u\right\Vert _{0,[1-\kappa p\mu _{1}^{-1},1]}\left\Vert \xi
\right\Vert _{E,I} \\
&\lesssim &\frac{\varepsilon _{2}}{\varepsilon _{1}^{1/2}}\frac{\varepsilon
_{1}^{1/2}}{\varepsilon _{2}^{1/2}}e^{-\beta p}\left\Vert \xi \right\Vert
_{E,I} \\
&\lesssim &e^{-\beta p}\left\Vert \xi \right\Vert _{E,I},
\end{eqnarray*}%
where (\ref{LeqNo}) was used. Therefore,%
\begin{equation*}
\left\vert \varepsilon _{2}\left\langle b\left( u-\mathcal{I}_{p}u\right)
^{\prime },\xi \right\rangle _{I}\right\vert \lesssim e^{-\beta p}\left\Vert
\xi \right\Vert _{E,I}
\end{equation*}%
and%
\begin{equation*}
\left\Vert \xi \right\Vert _{E,I}^{2}\lesssim e^{-\beta p}\left\Vert \xi
\right\Vert _{E,I}
\end{equation*}%
which completes the proof. \end{proof}

We conclude with the main result of the article.

\begin{theorem}
\label{main} Let $u$ be the solution of (\ref{de})--(\ref{bc}) and let $%
u_{FEM}\in S_{0}(\kappa ,p)$ be its approximation based on the Spectral
Boundary Layer Mesh. Then there exist a constant $\sigma >0$, independent of 
$\varepsilon _{1},\varepsilon _{2}$, such that%
\begin{equation*}
\left\Vert u-u_{FEM}\right\Vert _{E,I}\lesssim e^{-\sigma p}.
\end{equation*}
\end{theorem}

\begin{proof} We begin with the triangle inequality:%
\begin{equation*}
\left\Vert u-u_{FEM}\right\Vert _{E,I}\leq \left\Vert u-\mathcal{I}%
_{p}u\right\Vert _{E,I}+\left\Vert \mathcal{I}_{p}u-u_{FEM}\right\Vert
_{E,I},
\end{equation*}%
where $\mathcal{I}_{p}$ is the approximation operator of Theorem \ref%
{thm_sch}. \ The first term is handled by Lemma \ref{lemma_approx} and the
second by Lemma \ref{lemma_approx2}. \end{proof}


\section{Numerical results}

\label{nr}

In this section we present the results of numerical computations for two
examples, using the values 
\begin{equation}
\varepsilon _{1}=10^{-9},\varepsilon _{2}=10^{-4}\;;\;\varepsilon
_{1}=10^{-10},\varepsilon _{2}=10^{-5}\;;\;\varepsilon
_{1}=10^{-12},\varepsilon _{2}=10^{-12}\;,\;  \label{values}
\end{equation}%
(hence we cover all three regimes).

\vspace{0.5cm}

\textbf{Example 1:} We consider (\ref{de}), (\ref{bc}) with $%
b(x)=c(x)=f(x)=1 $. An exact solution is available, hence our results are
reliable. We take $\kappa =1$ in the definition of the mesh and we use
polynomials of degree $p=1,...,11$ for the approximation. Figure \ref{F1}
shows the percentage relative error measured in the energy norm, versus the
number of degrees of freedom $DOF=3p-1$, in a semi-log scale. The fact that
we see straight lines indicates the exponential convergence of the method,
while the robustness is visible since the method does not deteriorate as the
singular perturbation parameters tend to 0.

\begin{figure}[h]
\begin{center}
\includegraphics[width=0.6\textwidth]{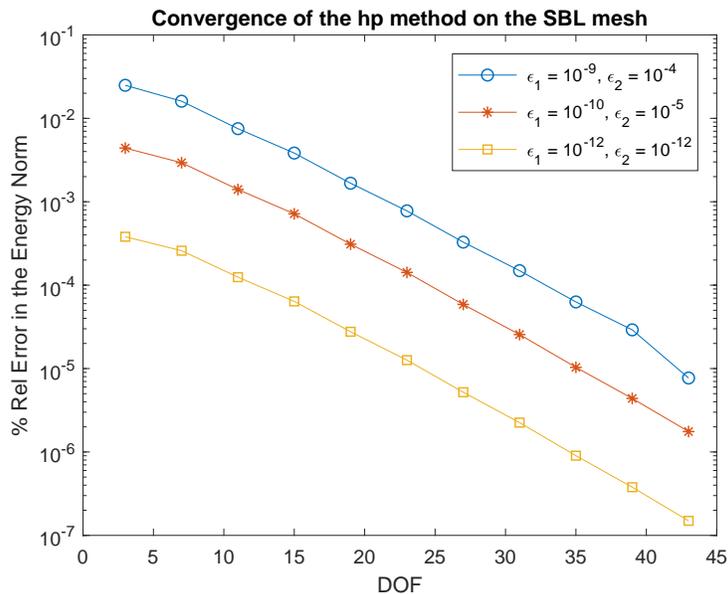}
\end{center}
\caption{Energy norm convergence for Example 1.}
\label{F1}
\end{figure}

In order to get a `clearer' picture of the performance of the method, we
show in Figures \ref{F2}--\ref{F4} the convergence in each regime
separately. In regime 1 ($\varepsilon _{1}<<\varepsilon _{2}^{2}$), we see
from Figure \ref{F2} that the method converges exponentially (we get
straight lines) and independently of $\varepsilon _{1},\varepsilon _{2}$
(the lines coincide). In regime 2 ($\varepsilon _{1}\approx \varepsilon
_{2}^{2}$), however, the lines do not coincide, even though we have
exponential convergence. This is due to the fact that the energy norm is not 
\emph{balanced} for reaction-diffusion problems (see, e.g. \cite%
{roos-franz11}, \cite{roos-schopf11}) and this manifests itself as the
method performing better as $\varepsilon _{1},\varepsilon _{2}\rightarrow 0$%
. The same is true in regime 3 ($\varepsilon _{1}>>\varepsilon _{2}^{2}$),
as seen in Figure \ref{F4}, since in this regime we again have a
reaction-diffusion problem.

\begin{figure}[h]
\begin{center}
\includegraphics[width=0.6\textwidth]{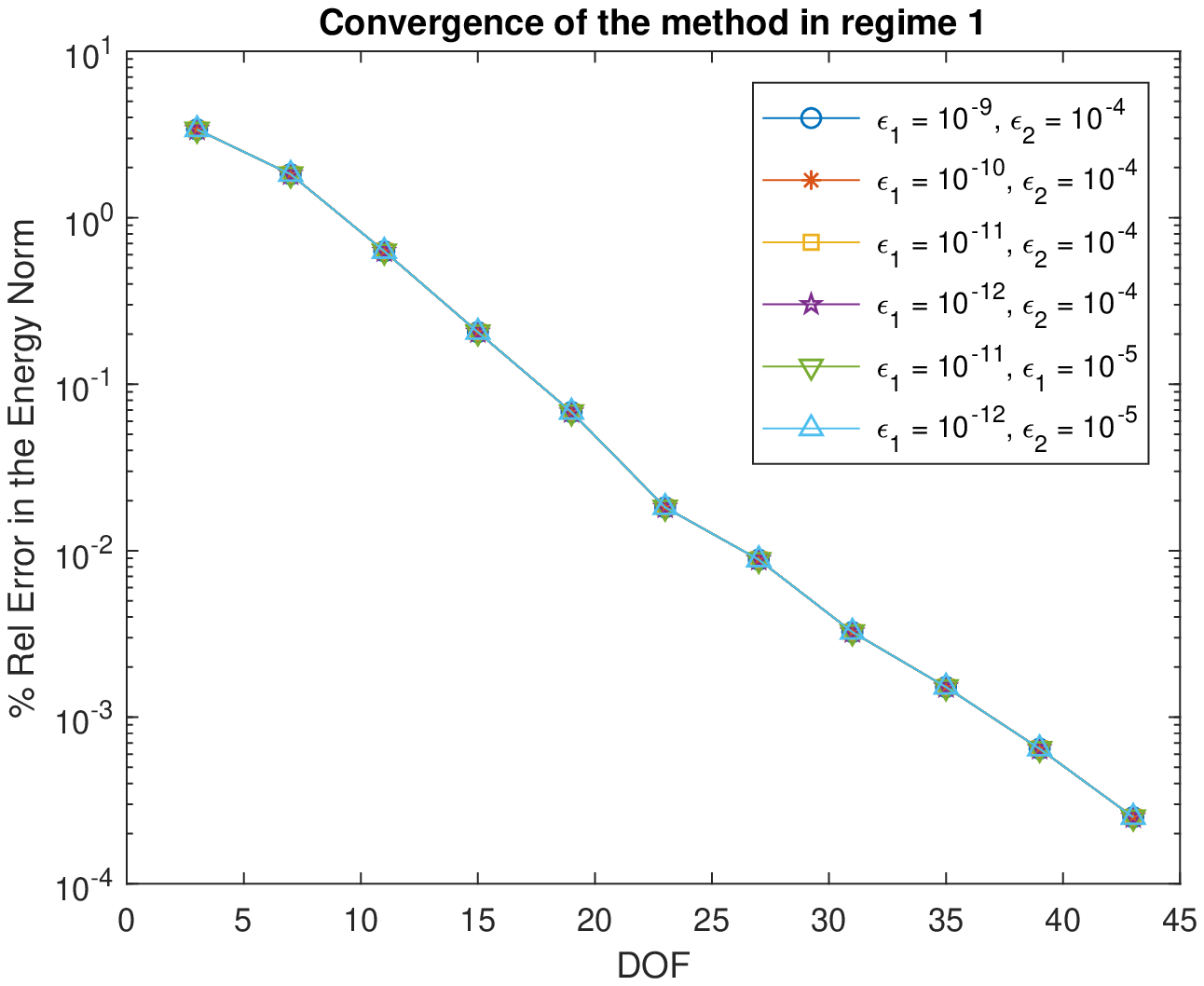}
\end{center}
\caption{Energy norm convergence for Example 1, when $\protect\varepsilon %
_{1}<<\protect\varepsilon _{2}^{2}$.}
\label{F2}
\end{figure}

\begin{figure}[h]
\begin{center}
\includegraphics[width=0.6\textwidth]{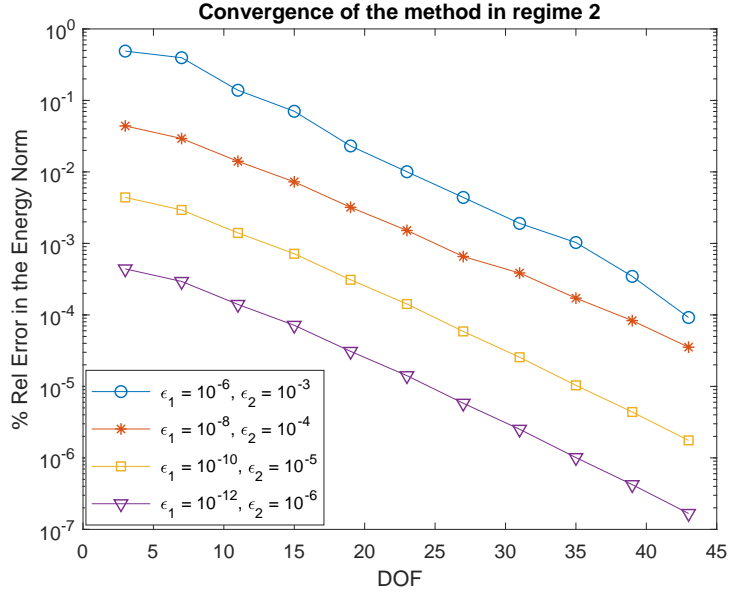}
\end{center}
\caption{Energy norm convergence for Example 1, when $\protect\varepsilon %
_{1}\approx \protect\varepsilon _{2}^{2}$.}
\label{F3}
\end{figure}

\begin{figure}[h]
\begin{center}
\includegraphics[width=0.6\textwidth]{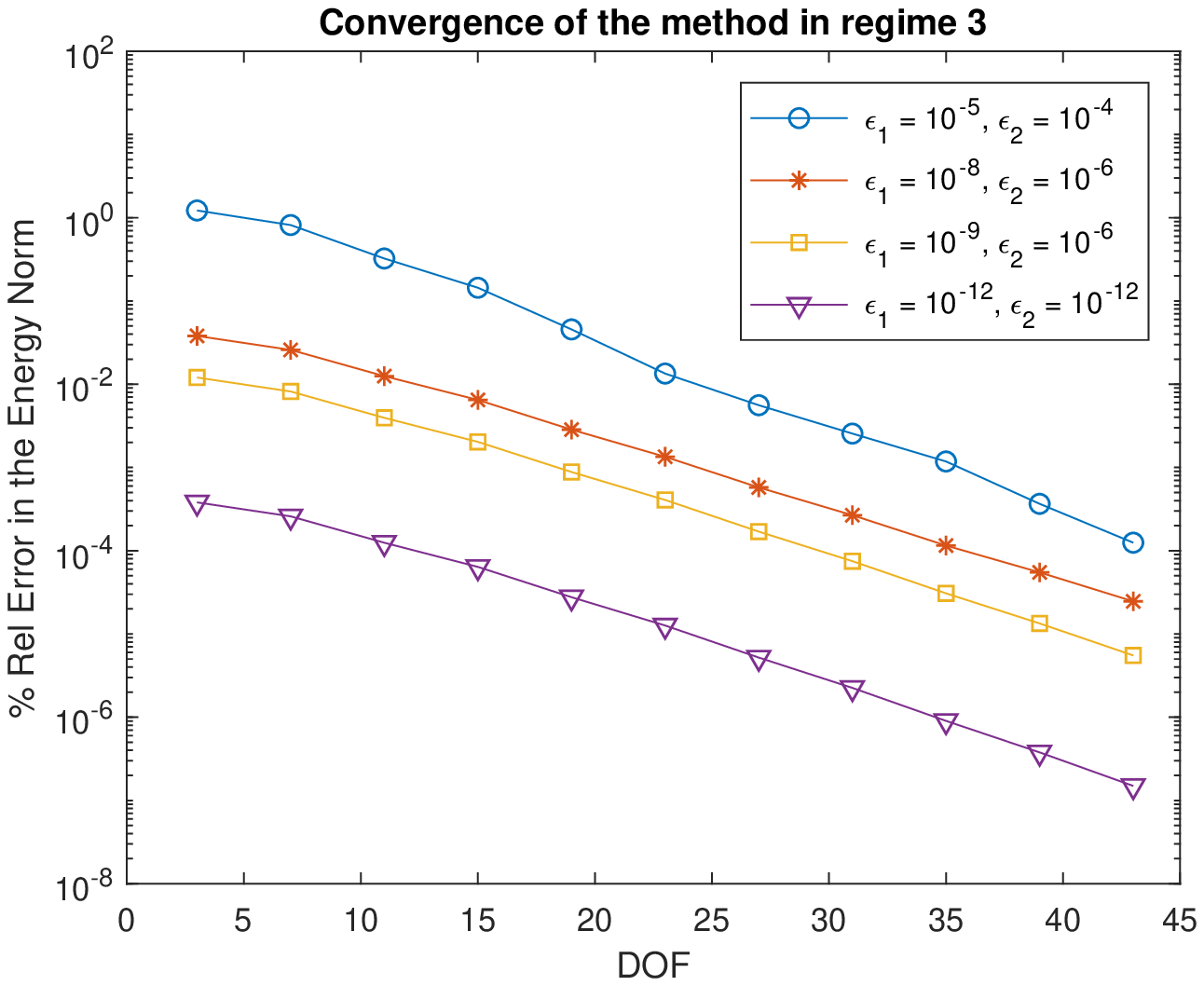}
\end{center}
\caption{Energy norm convergence for Example 1, when $\protect\varepsilon %
_{1}>>\protect\varepsilon _{2}^{2}$.}
\label{F4}
\end{figure}

\vspace{0.5cm}

\textbf{Example 2: }We now consider (\ref{de}), (\ref{bc}) with $%
b(x)=e^{x},c(x)=x,f(x)=1$. An exact solution is not available, so we use a
reference solution obtained with twice as many DOF. In Figure \ref{F5} we
show the convergence of the method for the values of $\varepsilon
_{1},\varepsilon _{2}$ given by (\ref{values}).

Once again we observe robust exponential convergence as $DOF$ is increased.

\begin{figure}[h]
\begin{center}
\includegraphics[width=0.6\textwidth]{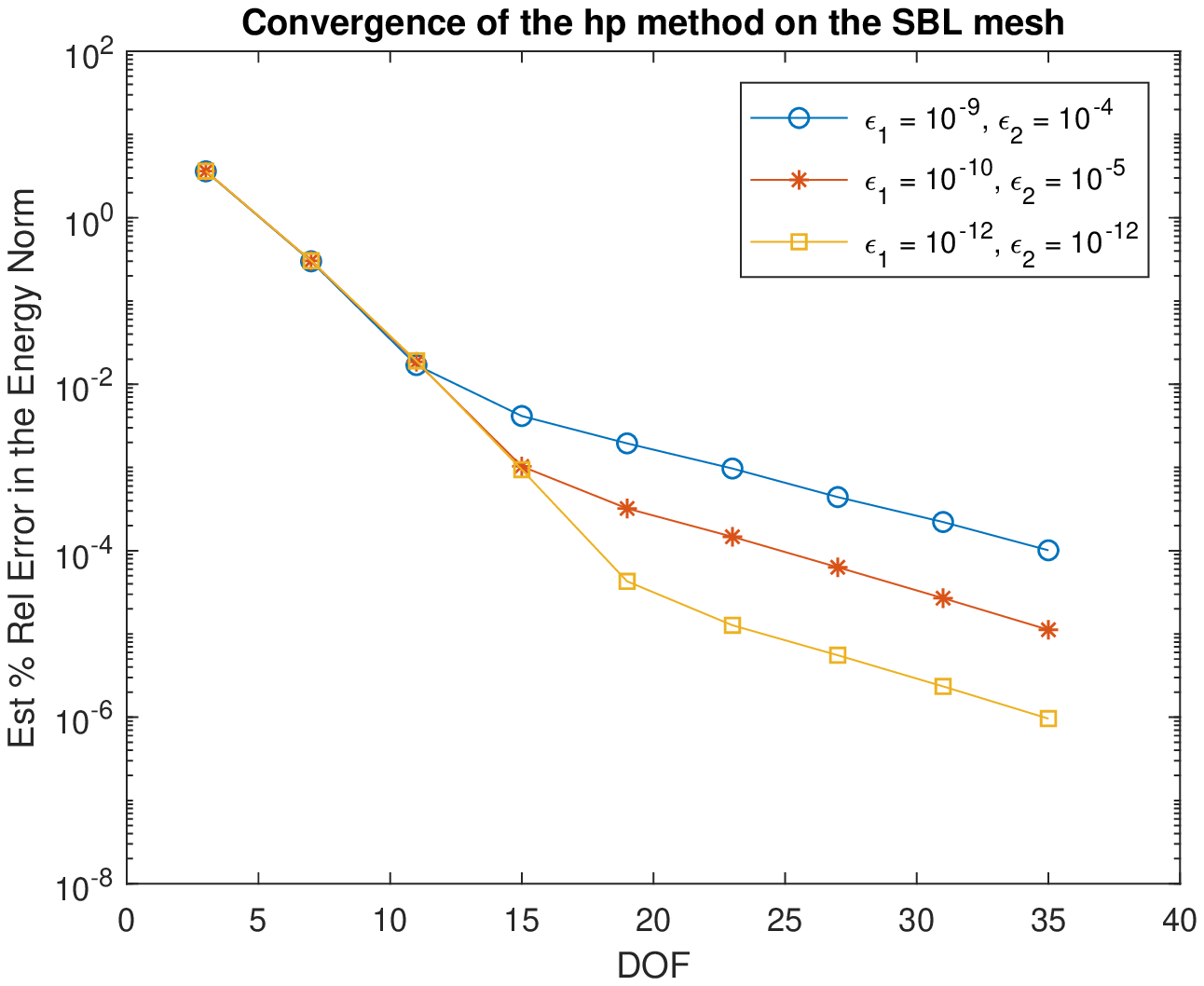}
\end{center}
\caption{Energy norm convergence for Example 2.}
\label{F5}
\end{figure}


\end{document}